\documentclass[11pt,leqno]{amsart}
\usepackage{amssymb,amsmath,amsthm,dsfont,fourier}

\usepackage{mathtools}
\usepackage{array}
\usepackage[text={6.5in,8.5in},centering]{geometry}
\usepackage[dvipsnames]{xcolor}
\usepackage{enumitem}
\usepackage{tikz-cd}
\usetikzlibrary{decorations.pathmorphing}
\usepackage[colorlinks=true,citecolor=NavyBlue,linkcolor=NavyBlue,urlcolor=NavyBlue]{hyperref}

\newcommand{\inv}{^{\raisebox{.2ex}{$\scriptscriptstyle-1$}}}

\newtheorem{theorem}{Theorem}[section]
\newtheorem{proposition}[theorem]{Proposition}
\newtheorem{lemma}[theorem]{Lemma}
\newtheorem{corollary}[theorem]{Corollary}
\theoremstyle{definition}
\newtheorem{definition}[theorem]{Definition}
\newtheorem{example}[theorem]{Example}
\newtheorem{remark}[theorem]{Remark}

\numberwithin{equation}{section}

\begin{document}

\author{Pubali Sengupta}
\address{
	Department of Mathematics, Jadavpur University, 188, Raja Subodh Chandra Mallick Rd,
	Jadavpur, Kolkata, West Bengal 700032, India.}
\email{pubalis.math.rs@jadavpuruniversity.in}
	
\author{Amartya Goswami}
\address{[1]  Department of Mathematics and Applied Mathematics, University of Johannesburg, P.O. Box 524, Auckland Park 2006, South Africa. [2]  National Institute for Theoretical and Computational Sciences (NITheCS), South Africa.}
\email{agoswami@uj.ac.za}

\author{Pronay Biswas}

\address{Department of Mathematics, Jadavpur University, 188, Raja Subodh Chandra Mallick Rd,
	Jadavpur, Kolkata, West Bengal 700032, India}
\email{pronayb.math.rs@jadavpuruniversity.in, pronaybiswas1729@gmail.com }
	
\author{Sujit Kumar Sardar}
\address{Department of Mathematics, Jadavpur University, 188, Raja Subodh Chandra Mallick Rd,
Jadavpur, Kolkata, West Bengal 700032, India.}
\email{sujitk.sardar@jadavpuruniversity.in}
	
\title{On the Subtractive Ideal Structure of Commutative Semirings}

\date{}




\begin{abstract}
In the theory of commutative semirings, the lack of additive inverses creates a structural divergence between ideals and congruences that does not exist in ring theory. The aim of this article is to  restore critical ideal-theoretic properties via the subtractive property. We first prove a subtractive analogue of Krull’s existence theorem, guaranteeing the existence of $k$-prime ideals disjoint from multiplicative sets. We show that in arithmetic semirings, the distinction between $k$-irreducible and $k$-strongly irreducible ideals vanishes, a coherence that we show is preserved under localisation. We investigate the structural properties and coincidence phenomena among associated subclasses of $k$-ideals in Laskerian semirings, von Neumann regular semirings, subtractive unique factorisation semidomains, principal ideal semidomains, and weakly Noetherian semirings. 
Finally, within the framework of additively idempotent semirings, we tether subtractive ideal-theoretic structures to underlying order-theoretic constraints, thereby obtaining new characterizations of $k$-prime and $k$-semiprime ideals. In that process, we also establish that every absolutely $k$-prime ideal is $k$-prime.
\end{abstract}
\maketitle
\tableofcontents 

\section{Introduction} 
Since Vandiver introduced \emph{rings without negatives} \cite{Vandiver34, Vandiver39}, the study of semirings has continued to evolve. Although semirings are often viewed as a natural generalisation of both rings and distributive lattices, they have yet to settle into a definitive position between these two theories.
Throughout, the term \emph{semiring} tacitly refers to a commutative semiring $(S,+,0,\cdot,1)$ with multiplicative identity $1$ and an element $0$, which is the additive identity and also multiplicatively \emph{absorbing}, that is, $a\cdot 0=0$, for all $a\in S$. 
In contrast to the ring-theoretic case, ideals and congruences in semirings are no longer in bijective correspondence. For example, take the semiring of non-negative integers, following Golan \cite{Golan}, we denote this by $(\mathds{N},+,0)$, and fix two natural numbers $a$, $b\in \mathds{N}\setminus \{0,1\}$.
Let $$\mathcal{D}:=\{(x,y)\in \mathds{N}\mid a\ \text{divides}\ |x-y|,\ \text{and}\ x>b,y>b\}.$$ Then $k=\mathcal{D}\cup \{(x,x)\mid x\in \mathds{N}\}$ is a congruence on $\mathds{N}$. Albeit the zeroth class $[0]_k$ is the singleton $\{(0,0)\}$, which shows that we cannot recover the original congruence from the zeroth class. \emph{De facto}, there are semirings which are ideal-simple but not congruence-simple, see \cite[Section 2]{Nam21}.  Therefore, in the sense of Gumm and Ursini \cite{GummUrsini}, semirings are not \emph{ideal determined}. As a generalisation of rings, semirings support a wider array of ideals (\textit{cf.} Definition \ref{Def: k-ideal and strong ideal}) and congruences (\textit{cf.} Definition \ref{Def: k-congruence}), including structures that are either trivial or do not arise in classical ring theory. For example, the class of $k$-ideals (also known as subtractive ideals) is trivial in rings, since each ideal of a ring is a $k$-ideal. Moreover, the class of \emph{strong ideals} in rings is incoherent. The ideal theory of semirings has been extensively studied; see \cite{Atani08, AtaniGarfami13, Nasehpour19, Iseki56, LaGrassa95, Goswami23, Goswami25, GoswamiDube24, GoswamiMaatallah25, BiswasBagSardarFilomat24}.

Despite its early appearance in an AMS notice \cite{Henriksen58}, Henriksen’s notion of a $k$-ideal has proved to be a persistent and productive idea in semiring theory. For a detailed account of the historical development of semiring theory, we refer the reader to G\l azek \cite{Glazek02}. Several authors, including Katsov, Abuhlail, Nam, and Il'in, have focused on characterizing algebraic properties of semirings admitting a rich class of $k$-ideals; see \cite{Il'inKatsovNam17, Abuhlail22, Abuhlail18, Katsov09, Katsov11}.
The study of the class of $k$-ideals appears in the earlier works \cite{LaTorre65, SenAdhikari92, SenAdhikari93, Han15, Han21, JunRayTolliver22, WeinertSenAdhikari96} and has been taken up again in more recent work \cite{GoswamiDube24}. This work is a sequel to \cite{GoswamiDube24}. We study the classes of $k$-prime, $k$-semiprime, $k$-primary, $k$-irreducible, and $k$-strongly irreducible ideals, together with their associated congruence classes, both in general semirings and across several distinguished classes of semirings. For the corresponding of ideals in rings we refer to \cite{AtiyahMacdonald, Azizi08, Schwartz16}.
The so-called ``exchange principal'' (\textit{cf.}\cite[Proposition 3.5]{JunRayTolliver22} and \cite[Proposition 3.11, Proposition 3.21]{GoswamiDube24}), namely
\begin{center}
 $k$-$X$-ideal $\iff $ $k$-ideal $+$ $X$-ideal,   
\end{center}
is true for $k$-prime, $k$-semiprime, $k$-irreducible and $k$-strongly irreducible ideals, which make these classes more tractable. 

Let us now retrace the main results established in this paper. 
In Section~\ref{Section 3: Some results on k-ideals on semirings}, we begin by examining examples and non-examples of subtractive and strongly subtractive semirings. We show that the ideal lattice of a strongly subtractive semiring is distributive, while that of a subtractive semiring is known to be modular. Consequently, strongly subtractive semirings provide a new source of examples of arithmetic semirings (see, Definition~\ref{Definition: arithmetic semiring}). In Proposition~\ref{krull's theorem}, we establish a subtractive analogue of Krull’s theorem for semirings, guaranteeing the existence of $k$-prime ideals. Although the classes of $k$-prime and $k$-primary ideals are shown to be distinct in Remark~\ref{Remark: ideal classes are different}, however, Theorem~\ref{loc14} shows that they coincide in von Neumann regular semirings. We also record a correspondence between $k$-strongly irreducible ideals and $i_k$-systems. The section concludes by proving that, within the class of arithmetic semirings, every $k$-irreducible ideal is in fact $k$-strongly irreducible, a phenomenon that does not occur in general semirings. 

Section~\ref{Section 4: k-ideals of quotient semirings} investigates the structural behaviour of distinguished subclasses of $k$-ideals under semiring homomorphisms and analyses their structural persistence in quotient semirings. We begin by establishing that the class of $k$-prime ideals is invariant under $k$-contraction. In Theorem~\ref{k-strongly irreducible ideals in quotient semirings}, we prove that $k$-strongly irreducible ideals descend to quotient semirings, and that the converse implication holds for arithmetic $k$-semirings. This result strengthens the corresponding theorem of Atani and Garfami; see \cite[Theorem~2.7]{AtaniGarfami13}. Finally, we show that the $k$-irreducibility of an ideal can be detected purely at the level of the quotient, namely, an ideal is $k$-irreducible if and only if the zero ideal of the associated quotient semiring is $k$-irreducible.

In Section~\ref{Section 5: k-ideals under localisation}, we investigate the behaviour of \(k\)-ideals and their subclasses under localisation of semirings. The central result, Theorem~\ref{loc}, extends Azizi’s correspondence to the semiring setting by establishing a bijection between \(k\)-proper strongly irreducible ideals of \(T\inv S\) and those \(k\)-proper strongly irreducible ideals of \(S\) that are disjoint from \(T\). We then analyse the stability of \(k\)-prime ideals under extension and contraction along semiring homomorphisms, describing the induced maps on \(k\)-spectra and their injectivity and surjectivity properties (Proposition~\ref{contraction of k-prime ideals}). Finally, we examine the behaviour of \(k\)-primary and \(k\)-strongly irreducible ideals under localisation and contraction (Propositions~\ref{loc6} and~\ref{loc7}). Taken together, these results show that localisation preserves and reflects essential ideal-theoretic structure in the subtractive framework.

The Section~\ref{Section 6: k-ideals and its variants in special semirings}, is devoted to a detailed analysis of $k$-ideals and their variants in several important classes of semirings, with particular emphasis on $k$-strongly irreducible ideals and their relationship with $k$-prime and $k$-primary ideals. The section is divided into two complementary parts.

In the first part, we study the behaviour of $k$-strongly irreducible ideals in Laskerian semirings, von Neumann regular semirings, subtractive unique factorisation semidomains, principal ideal semidomains, and weakly Noetherian semirings. Proposition \ref{loc3} establishes foundational links between $k$-strong irreducibility, $k$-primeness, and $k$-primary decomposition, showing in particular that in Laskerian and von Neumann regular semirings these notions coincide. In the setting of subtractive unique factorisation semidomains, Lemma \ref{loc15} and Theorem \ref{loc2} provide characterizations of $k$-strongly irreducible ideals in terms of prime power decompositions and least common multiples, and show that such ideals are necessarily $k$-primary. Localization properties of $k$-strongly irreducible ideals are analysed in Corollary \ref{loc4}, yielding correspondence results for unique factorisation semidomains and von Neumann regular semirings. Further structural consequences are obtained for principal ideal semidomains, where non-zero $k$-prime ideals are shown to be $k$-maximal, and for weakly Noetherian semirings, whereas Proposition \ref{weakly Noetherian semiring, every k-irreducible ideal is k-primary} demonstrates that every $k$-irreducible ideal is $k$-primary.

The subsection \ref{Subsection: additively idempotent} focuses on additively idempotent semirings, \textit{alias}, $B_1$-algebras. We give a new order-theoretic perspective of $k$-prime and $k$-semiprime ideals in Theorem \ref{k-prime via natural partial order} and Theorem \ref{k-semiprime via natural partial order}, respectively, using the natural partial order. For additively idempotent subtractive unique factorisation semidomains, Theorem \ref{k-strongly via order and lcm} characterizes $k$-strongly irreducible ideals through least common multiples. We then establish that saturated ideals (see Definition \ref{Def: excellent congruences and saturated ideals}) coincide precisely with $k$-ideals (Theorem \ref{saturated iff k-ideal}), clarifying the relationship between ideals and excellent congruences in this setting. We conclude the paper by showing that, in an additively idempotent semiring, every absolutely $k$-prime ideal is $k$-prime.

\section{Preliminaries}\label{prim} 
This section collects the basic definitions and results needed to keep the exposition self-contained. We adopt the standard terminology and foundational material as presented in Golan’s work \cite{Golan}. 
Unless otherwise stated, all semirings considered here are commutative with identity. A semiring homomorphism $f\colon S\rightarrow S'$ is a map satisfying $f(x+y)=f(x)+f(y)$, $f(xy)=f(x)f(y)$, $f(0_S)=0_{S'}$ and $f(1_S)=1_{S'}$. 

\begin{definition}[Page 2,\cite{Nasehpour19}] Suppose $S$ is a semiring.
\begin{enumerate}
\item A non-zero and non-unit element $r$ of  $S$ is said to be \emph{irreducible} if $r=r_1r_2$, where $r_1$, $r_2\in S$, then either $r_1$ or $r_2$ is a unit.

\item  An element $p$ of  $S$ is said to be \emph{prime} if the principal ideal $\langle p \rangle$ is prime.
        
\item A semiring $S$ is said to be a \emph{semidomain} if $ab=ac$ implies $b=c$, for all $b$, $c \in S$ and all non-zero $a \in S$.
        
        \item Let $S$ be a semidomain. Then $S$ is called a \emph{unique factorisation semidomain} if the following conditions hold:
        \begin{enumerate}
            \item Each irreducible element of $S$ is a prime element of $S$.
            \item Any non-zero, non-unit element of $S$ can be uniquely written as a product of irreducible elements of $S$.
        \end{enumerate}
    \end{enumerate}
\end{definition}

We follow the definitions of \emph{prime, semiprime} and \emph{maximal} ideals as presented in Golan \cite[Chapter 6, Chapter 7]{Golan}.

\begin{definition}[\cite{Henriksen58}, \cite{Golan}]\label{Def: k-ideal and strong ideal}
    \textit{An ideal $I$ of a semiring $S$ is said to be a \emph{$k$-ideal} if $x+y\in I$ and $y\in I$ implies $x\in I$. Moreover, $I$ is called \emph{strong} if $a+b\in I$ implies $a$, $b\in I$}.
\end{definition} 

\begin{remark}
    In literature $k$-ideals are also known as \emph{subtractive ideals} (see \cite[Chapter 6]{Golan}) and \emph{semistrong ideals} (see \cite[Section 1]{Chermynkh12}).
\end{remark}

It is straightforward to verify that every strong ideal is a $k$-ideal, though the converse need not hold. 

\begin{definition}[Definition 2.2, \cite{GoswamiDube24}]
    \textit{Let $S$ be a semiring $I$ be an ideal of $S$. The $k$-closure of $I$ is defined by
    \[\mathcal{C}_k(I) :=\{s\in S\mid s+x\in I\ \text{for some} \ x\in I\}.\]}
\end{definition}

We shall make frequent use of the following elementary properties of the
$k$-closure operator, which are consequences of \cite[Lemma~2.3]{GoswamiDube24}.

\begin{lemma}\label{lem:properties of k-closure}
	Let $I$ and $\{I_\lambda\}_{\lambda\in\Lambda}$ be ideals of a semiring $S$.
	Then:
	\begin{enumerate}
		\item $\mathcal{C}_k(I)$ is the least $k$-ideal of $S$ containing $I$.
		
		\item $\mathcal{C}_k(\mathcal{C}_k(I))=\mathcal{C}_k(I)$.
		
		\item For any family $\{I_\lambda\}_{\lambda\in\Lambda}$ of ideals,
		\[
		\mathcal{C}_k\left(\bigcap_{\lambda\in\Lambda} I_\lambda\right)
		\subseteq
		\bigcap_{\lambda\in\Lambda}\mathcal{C}_k(I_\lambda).
		\]
		Moreover, if each $I_\lambda$ is a $k$-ideal, then equality holds.
	\end{enumerate}
\end{lemma}

\begin{remark}\label{C_k fails}
	The inclusion in Lemma~\ref{lem:properties of k-closure}(3) cannot, in
	general, be reversed. Thus the $k$-closure operator need not preserve
	intersections, even for two ideals. The following small example makes this
	failure transparent.
	
	Let $S=\{0,a,b,1\}$ be the four-element chain
	\[
	0<a<b<1,
	\]
	with addition given by
	\[
	x+y=\max\{x,y\},
	\]
	and multiplication given by
	\[
	\begin{array}{c|cccc}
		\cdot & 0&a&b&1\\ \hline
		0&0&0&0&0\\
		a&0&0&0&a\\
		b&0&0&b&b\\
		1&0&a&b&1.
	\end{array}
	\]
	The ideal $I=\langle b\rangle=\{0,b\}$ is not a $k$-ideal, whereas
	$J=\langle a\rangle=\{0,a\}$ is a $k$-ideal. Since $I\cap J=\{0\}$, we
	obtain
	\[
	\mathcal{C}_k(I\cap J)=\{0\},
	\qquad
	\mathcal{C}_k(I)\cap\mathcal{C}_k(J)=\{0,a\}.
	\]
	Consequently,
	\[
	\mathcal{C}_k(I\cap J)
	\subsetneq
	\mathcal{C}_k(I)\cap\mathcal{C}_k(J).
	\]

	On the other hand, the situation improves considerably for von Neumann
	regular semirings. Recall that a semiring $S$ is said to be \emph{von Neumann regular} if for all $a\in S$ there exists $x\in S$ such that $a=a^2x$. In this setting, the $k$-closure does distribute over
	finite intersections of ideals. 
\end{remark}
Let us recall the definitions $k$-contraction and $k$-extension of $k$-ideals from \cite[Definition 4.1]{GoswamiDube24}.

\begin{definition} 
	Let $S$, and $S'$ be two semirings and $\varphi\colon S\to S'$ be a semiring homomorphism. 
	\begin{enumerate}
		\item If $J$ is a $k$-ideal of $S'$, then the \emph{$k$-contraction of} $J$, denoted by $J^c$, is defined by $\varphi\inv(J)$.
		\item If $I$ is a $k$-ideal of $S$, then the \emph{$k$-extension of} $I$, denoted by $I^e$, is defined by $\mathcal{C}_k(\langle\varphi(I)\rangle)$.
	\end{enumerate}    
\end{definition}

\begin{remark}
	By a $k$-contracted ideal $A$ of $S$, we mean, there exists a $k$-ideal $I$ of $S'$ such that $A=I^c$. Similarly, if $B$ is a $k$-extended ideal of $S'$, then $B=J^e$, for some $k$-ideal $J$ of $S$.
\end{remark}

We recall some subclasses of $k$-ideals that play a central role throughout this work.

\begin{definition}[Definition 3.4, Definition 3.20, \cite{GoswamiDube24}]
Let $S$ be a semiring. 
    \begin{enumerate}
        \item A $k$-ideal $I$ of $S$ is said to be \emph{$k$-strongly irreducible} if $A\cap B \subseteq I$ implies $A\subseteq I$ or $B\subseteq I$, for two $k$-ideals $A$, $B$ of $S$ and \emph{$k$-irreducible} if $A\cap B=I$ implies either $A=I$ or $B=I$.
        
        \item A $k$-ideal $I$ of $S$ is said to be \emph{$k$-prime} if $AB\subseteq I$ implies $A\subseteq I$ or $B\subseteq I$, for two $k$-ideals $A$, $B$ of $S$.
        
        \item If $I$ is a $k$-ideal of $S$, it is said to be \emph{$k$-semiprime} if $A^2\subseteq I$ implies $A\subseteq I$. 
        
        \item A $k$-ideal $I$ of $S$ is called \emph{$k$-primary} if $ab\in I$ and $a\notin I$ implies $b^n\in I$, for some $n\in \mathds{N}$.
    \end{enumerate}
\end{definition}

\begin{remark}\label{Remark: ideal classes are different} Despite their close relationship, the classes of $k$-prime, $k$-semiprime, and $k$-strongly irreducible ideals do not coincide in general semirings. We note two instances where this divergence occurs. 
\begin{enumerate}
    \item Clearly, every $k$-prime ideal is $k$-strongly irreducible and $k$-semiprime. The converse is not always true, like in $(\mathds{N},+,\cdot)$, the ideal $4\mathds{N}$ is $k$-strongly irreducible but not $k$-prime. Moreover, the ideal $6\mathds{N}$ in $\mathds{N}$ is a $k$-semiprime ideal, which is neither $k$-prime nor $k$-strongly irreducible. Moreover, the ideal $p^2\mathds{N}$, where $p$ is any prime number, is a $k$-primary ideal which is not $k$-prime. 
    \item Each $k$-strongly irreducible ideal is $k$-irreducible, whereas the converse is not true in general. Consider the semiring $\mathds{N}[x,y]$ and the ideal $\langle x,y\rangle$. Then $\langle x,y\rangle$ is a $k$-ideal, and it cannot be written as an intersection of two proper $k$-ideals, that is, it is $k$-irreducible. If $I=\langle x,y^2\rangle$ and $J=\langle x^2,y\rangle$, then $\langle x,y^2\rangle \cap \langle x^2,y\rangle \subseteq \langle x,y\rangle$. But $\langle x,y^2\rangle \nsubseteq \langle x,y\rangle$ and $\langle x^2,y\rangle \nsubseteq \langle x,y\rangle$. 
\end{enumerate}
    We also note that not all prime, semiprime, and strongly irreducible ideals are $k$-ideals in general semirings. For instance, take the unique maximal ideal $\mathds{Q}_+[X]\setminus \mathds{Q}_+$ in the semiring $\mathds{Q}_+[X]$, which is not a $k$-ideal. 
\end{remark}

We denote the collection of all prime ideals of a semiring $S$ by $\mathrm{Spec}(S)$, and the collection of all $k$-prime ideals by $\mathrm{Spec}_k(S)$. 

\begin{definition}[Definition 3.12, \cite{GoswamiDube24}]
    The \emph{$k$-radical} of a $k$-ideal $I$ of a semiring $S$ is defined as \[\mathcal{R}_k(I):=\bigcap_{I\subseteq P}\{P\mid P\in \mathrm{Spec}_k(S)\}.\]
\end{definition}

 Throughout the paper, we repeatedly invoke the following elementary properties of the $k$-radical closure (lifted from \cite[Lemma 3.13]{GoswamiDube24}) whenever required.

\begin{lemma} In the following, $I$ and $J$ are ideals of a semiring $S$.
\begin{enumerate}
    \item $\mathcal{R}_k(I)$ is a $k$-ideal containing $I$.
    \item $\mathcal{R}_k(IJ)=\mathcal{R}_k(I\cap J)=\mathcal{R}_k(I)\cap \mathcal{R}_k(J)$.
\end{enumerate}
\end{lemma}
Recall that a \emph{congruence} on a semiring $S$ is an equivalence relation
$\rho$ on $S$ such that
\[
a\,\rho\,b,\quad c\,\rho\,d
\quad\Longrightarrow\quad
a+c\,\rho\,b+d
\quad\text{and}\quad
ac\,\rho\,bd
\]
for all $a,b,c,d\in S$. Equivalently, $\rho$ is a subsemiring of
$S\times S$ which is an equivalence relation on $S$.

For an ideal $I$ of $S$, the associated \emph{Bourne congruence} is the
relation $K_I$ on $S$ defined by
\[
a\ K_I b
\quad\Longleftrightarrow\quad
a+i=b+j
\quad\text{for some }i,j\in I.
\]
Equivalently,
\[
K_I
=
\{(a,b)\in S\times S:
a+i=b+j\text{ for some }i,j\in I\}.
\]
It is readily verified that $K_I$ is a congruence on $S$.
Note that the appropriate congruence theoretic counterpart of $k$-ideals is the $k$-congruences, in the sense of Han \cite[Definition 3.2]{Han21}. 

\begin{definition} \label{Def: k-congruence}
    Let $S$ be a semiring. A congruence $\rho$ on $S$ is said to be a \emph{$k$-congruence} if $\rho=K_I$, for some ideal $I$ of $S$.

\end{definition}

There exists a one-one correspondence between $k$-ideals and $k$-congruences, see \cite[Theorem 3.8]{Han21}.

\begin{definition}
    A $k$-congruence $\theta$ is said to \emph{irreducible} if for any two $k$-congruence $\rho_1$, $\rho_2$, and $\rho_1\cap \rho_2=\theta$, implies that $\rho_1=\theta$ or $\rho_2=\theta$.
\end{definition}

If $S$ is an additively idempotent semiring, a natural partial order on $S$ is defined as follows: for $x$, $y\in S$, \[x\leqslant y \Longleftrightarrow x+y=y.\]
Then, we have the following from \cite[Proposition 3.11]{JunRayTolliver22}. 

\begin{proposition}
    Let $S$ be an additively idempotent semiring and $I$ be an ideal. Then $I$ is a $k$-ideal if and only if for all $x\in I$ and $y\leqslant x$, we have $y\in I$.
\end{proposition}

\section{Basic results}\label{Section 3: Some results on k-ideals on semirings}
 
The ideal theory of a semiring has a remarkably different flavour from that of a ring. While the lattice of ideals of a ring is always modular, this property fails in general for semirings; see \cite[Example 6.36]{Golan}. In contrast, $k$-ideals retain many of the familiar properties of ring ideals, for example, the lattice of all $k$-ideals with a meaningful join and meet is modular; see \cite[Theorem 2.4]{SenAdhikari92}. A semiring can be $k$-ideal simple, that is, the only $k$-ideals are the trivial ones. We give some examples of $k$-ideal simple semirings. 

\begin{example}
\begin{enumerate}
    \item[]
        \item The semiring $\mathds{Q}_+[\sqrt{2}]$ with usual addition and multiplication is $k$-ideal simple; see \cite[Section 5.2]{BiswasGoswamiSardar}. In general, a semiring $S$ satisfying the condition:
        \[
        1+xa=ya,\ \text{for all}\ a\in S\setminus\{0\}\ \text{and for some}\ x,y\in S,
        \] has trivial $k$-ideals. For a proof, see \cite[Corollary 2.8]{SenAdhikari93}.
        
        \item  Consider the semiring $C^\infty(X)=\{f\colon X\rightarrow (0, \infty]\mid $  $f$ $ \text{is continuous}\}\cup \{0\}$ with pointwise addition and pointwise multiplication, where $X$ is a topological space and the constant function $0$ which is the additive identity and an absorbing element.  Since $1+\infty=\infty$ and the constant function $\infty$ is also multiplicatively absorbing,  then $\infty$ belongs to every ideal, but $1$ is not. Therefore, the semiring $C^\infty(X)$ is a $k$-ideal simple semiring. 
    \end{enumerate}
\end{example}

\begin{definition}
    A semiring $S$ is said to be a \emph{$k$-semiring} if each ideal of $S$ is a $k$-ideal. Moreover, $S$ is called \emph{strongly $k$-semiring}  if each ideal of $S$ is a strong ideal. 
\end{definition}

In literature, $k$-semirings and strongly $k$-semirings are also known as \emph{subtractive semirings} and \emph{strongly subtractive}, see \cite{Katsov09, Katsov11, Il'inKatsovNam17}. There are numerous examples of $k$-semirings; we list a few of them.

\begin{example}\label{example: subtractive semirings}
\begin{enumerate}
    \item[] 
        \item Any bounded distributive lattice $(L,\vee,\wedge,0,1)$ when considered as a semiring. In fact, it is a strongly $k$-semiring.
        \item Product of rings and bounded distributive lattices, both considered as semirings.
        \item The semiring of non-negative real-valued continuous functions over an $F$-space, see \cite[Theorem 3.9]{BiswasBagSardarFilomat24}
        \item The semiring of non-negative real-valued measurable functions over a measurable space; see \cite[Corollary 3.2]{BiswasBagSardarFilomat25}.
    \end{enumerate}
\end{example}

\begin{remark}
    The semirings in (3) and (4) of Example \ref{example: subtractive semirings} are cancellative. In contrast, there are semirings which are cancellative, but not every ideal is a $k$-ideal. For example, take the semiring $C^+(X)$ for any Tychonoff space $X$, Then with pointwise addition and multiplication, $C^+(X)$ is a cancellative semiring, but not all ideals are $k$-ideals; see \cite[Example 3.1]{BiswasBagSardarFilomat24}.
\end{remark}

While $k$-ideals recover much of the behaviour of ring ideals, the next theorem shows that strong ideals align naturally with lattice ideals. We denote the collection of all ideals of a semiring $S$ by $\mathrm{Id}(S)$. 

\begin{theorem}\label{Ideal lattice of strongly subtractive is distributive}
    Let $S$ be a strongly subtractive semiring. Then the lattice $(\mathrm{Id}(S),+,\cap)$ is distributive. 
\end{theorem}

\begin{proof}
    It is enough to show that $I\cap (J+K)\subseteq (I\cap J)+(I\cap K)$ for any $I$, $J$ and $K$ in $\mathrm{Id}(S)$. Let $a\in I\cap (J+K)$. Then $a=x+y\in I$, for some $x\in J$ and $y\in K$. Since $I$ is strong, $x\in I$ and $y\in I$. We conclude that $a\in (I\cap J)+(I\cap K)$ and hence $I\cap (J+K)\subseteq (I\cap J)+(I\cap K)$.
\end{proof}

As a corollary, we get that the ideal lattice of any bounded distributive lattice is distributive. 
Next, we note an elementary fact, which will be used freely in what follows.
\begin{lemma}\label{loc16}
    Every $k$-maximal ideal of a semiring is $k$-prime.
\end{lemma}
    
\begin{proof}
    Let $S$ be a semiring and let $M$ be a $k$-maximal ideal of $S$. Let $ab \in M$ and $a \notin M$, for some $a$, $b \in S$. Then $M \subseteq M+\langle a\rangle$, which further implies that $M \subseteq M+\langle a\rangle  \subseteq \mathcal{C}_k(M+\langle a\rangle)$. Since $M$ is a $k$-maximal ideal, we must have $\mathcal{C}_k(M+\langle a\rangle)= S$. Therefore $1 \in \mathcal{C}_k(M+\langle a\rangle)$. By definition, $1+ ax_1+m_1 = ax_2+m_2$, for some $x_1$, $x_2 \in S$ and $m_1$, $m_2 \in M$. We get $$b+abx_1+bm_1 = abx_2+bm_2.$$ So, $b \in \mathcal{C}_k(M)=M$ as $M$ is a $k$-ideal. Hence, $M$ is a  $k$-prime ideal. 
\end{proof}

\begin{proposition}
    Let $S$ be a semiring and $x\in S$. Then the following are equivalent.
    \begin{enumerate}
        \item $1_{S}\in \mathcal{C}_k(\langle x \rangle)$.
        \item $x$ cannot be in any $k$-maximal ideal of $S$.
    \end{enumerate}
\end{proposition}

\begin{proof}
    Let $x$ be an element that does not belong to any $k$-maximal ideal of $S$. Since every proper $k$-ideal is contained in a $k$-maximal ideal of $S$, the ideal $\mathcal{C}_k(\langle x \rangle)$ cannot be a proper $k$-ideal. Therefore, $\mathcal{C}_k(\langle x \rangle)=S$ and so $1_{S}\in \mathcal{C}_k(\langle x \rangle)$. Conversely, let $1_{S} \in \mathcal{C}_k(\langle x \rangle)$. Let $I$ be a $k$-maximal ideal of $S$ such that $x\in I$. Then $\langle x \rangle\subseteq I$, and hence, $\mathcal{C}_k(\langle x \rangle) \subseteq \mathcal{C}_k(I)=I$. From this, we conclude that $1_{S} \in I$, a contradiction. Hence $x\notin I$.
\end{proof}

As in the case of rings, the correspondence between prime ideals and $m$-systems is well established in semiring theory; see \cite[Corollary 7.11]{Golan}. Golan further proves a Krull-type converse in \cite[Proposition 7.12]{Golan}. The following proposition provides a subtractive analogue of Krull’s result for semirings. 

\begin{proposition}\label{krull's theorem}
    Let $T$ be a multiplicatively closed set of a semiring $S$ and let $I$ be a $k$-ideal maximal among the $k$-ideals disjoint from $T$. Then $I$ is $k$-prime.
\end{proposition}

\begin{proof}
 To prove $I$ is $k$-prime, we need to show that $I$ is prime. Let $ab \in I$ such that $a$, $b \notin I$. Consider the $k$-ideal $\mathcal{C}_k(I+\langle a\rangle)$. Then $\mathcal{C}_k(I+\langle a\rangle)$ is a $k$-ideal properly containing $I$, and therefore, intersects $T$. So, there exists $s_1 \in \mathcal{C}_k(I+\langle a\rangle)$, which implies $s_1+i_1+x_1a=i_2+x_2a$, for some $i_1$, $i_2 \in I$ and $x_1$, $x_2 \in S$. Then $$s_1b+i_1b+x_1ab=i_2b+x_2ab.$$ This implies $s_1b \in I$ as $i_1$, $i_2$, $ab \in I$. Similarly, $\mathcal{C}_k(I+\langle b\rangle)$ is a $k$-ideal properly containing $I$, and so, it intersects $T$. Therefore, there exists $s_2 \in \mathcal{C}_k(I+\langle b\rangle)$. Then, $s_2+j_1+y_1b=j_2+y_2b$, for some $j_1$, $j_2 \in I$ and $y_1$, $y_2 \in S$. So, $s_2a+j_1a+y_ab=j_2a+y_2ab$, which implies $s_2a \in I$, as $j_1$, $j_2$, $ab \in I$. We have, \[(s_1+i_1+x_1a)(s_2+j_1+y_1b)=(i_2+x_2a)(j_2+y_2b).\] Thus \[s_1s_2+s_1+s_1j_1+y_1s_1b+i_1(s_2+j_1+y_1b)+x_1s_2a+j_1x_1a+x_1y_1ab=i_2(j_2+y_2b)+j_2x_2a+x_2y_2ab.\] Therefore, we get $s_1s_2 \in I$, which is a contradiction to the fact that $I$ is disjoint from $T$. Hence $I$ is $k$-prime.
\end{proof}
Although $k$-primary and $k$-prime ideals are distinct in general (\textit{c.f.} Remark \ref{Remark: ideal classes are different}), the following theorem shows that this distinction disappears for von Neumann regular semirings. 

\begin{theorem}\label{loc14}
    In a von Neumann regular semiring, the classes of $k$-prime and $k$-primary ideals coincide.
\end{theorem}

\begin{proof}
    Let $S$ be a von Neumann regular semiring and $I$ be a $k$-prime ideal of $S$. Then $I$ is a $k$-ideal of $S$. Therefore, $xy \in I$ implies $x \in I$ or $y \in I$, which further implies $I$ is $k$-primary.
    Conversely, suppose that $I$ is a $k$-primary ideal of $S$. Let $xy \in I$. Then $x \in I$ or $y^n \in I $, for some natural number $n$. If $x \in I$, then $I$ is prime. Now, let $y^n \in I$. Since $S$ is von Neumann regular for $y \in S$, there exists $x \in S$ such that $y=y^2x$. This implies $$y^{n-2} y = y^{n-2}y^2x,$$ which further implies $y^{n-1}=y^nx \in I$, as $I$ is an ideal of $S$. Again, $y^{n-3} y=y^{n-3} y^2x$ implies $y^{n-2}=y^{n-1}x \in I$. Continuing in this way, we have $y^2 \in I$. Thus, $y=y^2x \in I$.
\end{proof}

Motivated by the correspondence between strongly irreducible ideals and $i$-systems \cite[Proposition 7.33]{Golan}, we introduce a subtractive variant of the notion of an $i$-system.

\begin{definition}
    A non-empty subset $A$ of a semiring $S$ is an \emph{$i_k$-system} if $a$, $b \in A $ implies that $\mathcal{C}_k(\langle a\rangle) \cap \mathcal{C}_k(\langle b\rangle) \cap A \neq \emptyset$.
\end{definition}

\begin{theorem}\label{loc11}
    Let $I$ be a $k$-ideal of a semiring $S$. Then the following conditions are equivalent.
    \begin{enumerate}
        \item $I$ is $k$-strongly irreducible.
        \item If $a$, $b \in S$ satisfying $\mathcal{C}_k(\langle a\rangle) \cap \mathcal{C}_k(\langle b\rangle) \subseteq I$, then $a \in I$ or $b \in I$.
        \item $S\backslash I$ is an $i_k$-system.
    \end{enumerate}
\end{theorem}

\begin{proof}
    (1)$\Rightarrow$(2): Clear from the definition.

    (2)$\Rightarrow$(3): Suppose $a$, $b \in S\setminus I$ and if possible, let $\mathcal{C}_k(\langle a\rangle) \cap \mathcal{C}_k(\langle b\rangle) \cap (S \setminus I)= \emptyset$. This implies $\mathcal{C}_k(\langle a\rangle) \cap \mathcal{C}_k(\langle b\rangle) \subseteq I$. So, $a \in I$ or $b \in I$, from $(2)$, which is a contradiction. Therefore, $\mathcal{C}_k(\langle a\rangle) \cap \mathcal{C}_k(\langle b\rangle) \cap (S \setminus I) \neq \emptyset$. Hence, $S \setminus I$ is an $i_k$-system.

    (3)$ \Rightarrow $(1):
    Let $A$, $B$ be two $k$-ideals of $S$ such that $A \cap B \subseteq I$, but $A \nsubseteq I$ and $B \nsubseteq I$. Then, there exist $x \in A \backslash I$ and $y \in B \backslash I$. By $(3)$, as $S \setminus I$ is an $i_k$-system, there exists $z \in \mathcal{C}_k(\langle x\rangle) \cap \mathcal{C}_k(\langle y\rangle) \cap (S \setminus I)$. This further implies $z \in A \cap B$ such that $z \notin I$, a contradiction. So, we get $A \subseteq I$ or $B \subseteq I$. Thus $I$ is $k$-strongly irreducible.
\end{proof}

As observed in Remark~\ref{Remark: ideal classes are different}, $k$-strongly irreducible and $k$-irreducible ideals may differ in general. In the sequel, we show that they coincide for a special class of semirings, after introducing the necessary definition.

\begin{definition}\label{Definition: arithmetic semiring}
    A semiring $S$ is said to be an \emph{arithmetic semiring} if 
        $A+(B \cap C)= (A+B) \cap (A+C)$, or equivalently,  
        $A \cap (B+C)= (A \cap B)+(A \cap C)$, for any ideals $A$, $B$, $C$ of $S$. In other words, in an arithmetic semiring, the set of all ideals of  $S$ forms a distributive lattice.
    
\end{definition}

\begin{proposition}
    The following are equivalent for a semiring $S$.
    \begin{itemize}
        \item[1.] $S$ is an arithmetic semiring.
        \item[2.] For any $a$, $b$, $c\in S$, $\langle a\rangle\cap (\langle b\rangle +\langle c\rangle)=(\langle a\rangle \cap \langle b\rangle)+(\langle a\rangle\cap \langle c\rangle)$ 
    \end{itemize}
\end{proposition}

\begin{proof}
    (1)$\Rightarrow$(2): Obvious.

    (2)$\Rightarrow$(1): Suppose (2) holds. Let $I$, $J$, and $K$ be any elements of $\mathrm{Id}(S)$. Then $(I\cap J)+(I\cap K)\subseteq I\cap (J+K)$ is always true in any lattice. Let $x\in I\cap (J+K)$. Therefore $x=i$ and $x=j+k$, for some $i\in I$, $j\in J$, and $k\in K$. Then, by the hypothesis, we have 
    \[x\in \langle i \rangle \cap (\langle j \rangle + \langle k \rangle) = (\langle i \rangle \cap \langle j \rangle)+(\langle i \rangle \cap \langle k \rangle) \subseteq (I\cap J)+(I\cap K).\]
    Hence $I\cap (J+K)\subseteq (I\cap J)+(I\cap K)$.
\end{proof}

From Theorem \ref{Ideal lattice of strongly subtractive is distributive}, we know that all strongly subtractive semirings are arithmetic. In the next theorem, we show that the classes of $k$-strongly irreducible and $k$-irreducible ideals coincide in arithmetic $k$-semirings. 
\begin{theorem}
    Let $S$ be an arithmetic $k$-semiring. A proper ideal $I$ of $S$ is $k$-strongly irreducible if and only if it is $k$-irreducible.
\end{theorem}

\begin{proof}
Every $k$-strongly irreducible ideal is $k$-irreducible. Conversely, let $I$ be $k$-irreducible and let $A,B$ be $k$-ideals with $A\cap B\subseteq I$. Since $S$ is a $k$-semiring, $I+A$ and $I+B$ are $k$-ideals. Arithmeticity gives
\[
(I+A)\cap(I+B)=I+(A\cap B)=I.
\]
By $k$-irreducibility, $I+A=I$ or $I+B=I$. Hence $A\subseteq I$ or $B\subseteq I$, proving that $I$ is $k$-strongly irreducible.
\end{proof}

\section{On quotient semirings}\label{Section 4: k-ideals of quotient semirings}
In this section, we investigate the behaviour of $k$-ideals, $k$-irreducible ideals, and $k$-strongly irreducible ideals under semiring homomorphisms and passage to quotient semirings. 
The following proposition shows that $k$-prime ideals are stable under $k$-contractions.

\begin{proposition}
    Let $S$ and $S'$ be two semirings and $\varphi\colon S\to S'$ be a semiring homomorphism. If $J$ is a $k$-prime ideal of $S'$, then the $k$-contraction of $J$, $J^c$ is a $k$-prime ideal of $S$. 
\end{proposition}

\begin{proof}
    
    We only prove the `prime part'. Let $xy\in J^c=\varphi\inv(J)$. This implies $\varphi(xy)\in J$, which further implies $\varphi(x)\varphi(y)\in J$. Then $\varphi(x)\in J$ or $\varphi(y)\in J$. Therefore $x\in \varphi\inv(J)=J^c$ or $y\in \varphi\inv(J)=J^c$. Thus $J^c$ is a $k$-prime ideal of $S$. 
\end{proof}

In contrast with $k$-prime ideals, $k$-strong irreducible ideals are not automatically preserved by contraction along a general surjective semiring homomorphism. In particular, such a contraction statement for $k$-strongly irreducible ideals depends on the distributivity of $\mathcal{C}_k$ with the intersections; which is not true general as seen in Remark~\ref{C_k fails}.

    Let $S$ be a semiring and $K_I$ be the Bourne congruence corresponding to the ideal $I$ of $S$. The quotient $S/K_I$, henceforth denoted by $S/I$, has equivalence classes of the form $r+I$. Then $S/I$ is a semiring with respect to the operations $(r_1+I)+(r_2+I)=r_1+r_2+I$ and $(r_1+I)(r_2+I)=r_1r_2+I$.

The following lemmas summarise the basic properties of ideals and $k$-ideals in quotient semirings; we mention them without proof.

\begin{lemma}\label{loc12}
    Let $S$ be a semiring. Then the following hold:
    \begin{enumerate}
        \item If $a\in I$, then $a+I=I$, for all ideals $I$ of $S$.
        \item If $I$ is a $k$-ideal of $S$, then $c+I=I$ if and only if $c \in I$.
        \item $I$, $J$ are ideals of $S$ such that $I\subseteq J$, then $J/I$ is an ideal of $S/I$. In particular, if $J$ is a $k$-ideal of $S$, then $J/I$ is a $k$-ideal of $S/I$.
        \item If $1+I\in J/I$, for ideals $I$, $J$ with $I \subseteq J$, then $S/I =J/I$.
    \end{enumerate}
\end{lemma}

\begin{lemma}\label{loc13}
    Let $I$ be an ideal of a semiring $S$. If $J$, $K$, $L$ are $k$-ideals containing $I$, then $(J/I)\cap (K/I)=L/I$ if and only if $J\cap K=L$.
\end{lemma}

In \cite[Theorem 2.7]{AtaniGarfami13}, Atani and Garfami established one direction of the correspondence for $k$-strongly irreducible ideals in quotient semirings. The following theorem strengthens this result by proving the converse in the setting of arithmetic semirings.
\begin{theorem}\label{k-strongly irreducible ideals in quotient semirings}
    Let $S$ be a semiring, let $I$ be an ideal of $S$, and let $J$ be a $k$-strongly irreducible ideal of $S$ containing $I$. Then $J/I$ is a $k$-strongly irreducible ideal of $S/I$. Conversely, if $S$ is an arithmetic $k$-semiring and $J/I$ is $k$-strongly irreducible in $S/I$, then $J$ is $k$-strongly irreducible in $S$.
\end{theorem}

\begin{proof}
    For the first implication, see \cite[Theorem 2.7]{AtaniGarfami13}. For the converse, let $A,B$ be $k$-ideals of $S$ such that $A\cap B\subseteq J$. Because $S$ is a $k$-semiring, $A+I$ and $B+I$ are $k$-ideals. Arithmeticity yields
\[
(A+I)\cap(B+I)=I+(A\cap B)\subseteq J.
\]
Hence
\[
((A+I)/I)\cap((B+I)/I)\subseteq J/I.
\]
The $k$-strong irreducibility of $J/I$ implies $(A+I)/I\subseteq J/I$ or $(B+I)/I\subseteq J/I$. In the first case, if $a\in A$, then $a+I=j+I$ for some $j\in J$, so $a+i_1=j+i_2$ for suitable $i_1,i_2\in I\subseteq J$. Since $J$ is subtractive, $a\in J$. Thus $A\subseteq J$; similarly, the second case gives $B\subseteq J$.
\end{proof}

The following result is significant in the sense that it shows $k$-irreducibility is preserved and reflected under passage to quotients, allowing the study of $k$-irreducible ideals to be reduced to the zero ideal in an appropriate quotient semiring.

\begin{theorem}
    Let $I$ be a $k$-ideal of a semiring $S$. Then $I$ is $k$-irreducible if and only if the zero ideal $\langle \overline{0}\rangle$=$I/I$ in $S/I$ is $k$-irreducible.
\end{theorem}

\begin{proof}
    Suppose, $I$ is a $k$-irreducible ideal of $S$ and $I\neq S$. Then, there exists $a \in S$ such that $a+I\neq I$. As, $I$ is a $k$-ideal, the zero ideal $\langle \overline{0}\rangle$ is also a $k$-ideal, by Lemma~\ref{loc12}. Let $J/I \cap H/I = \langle \overline{0}\rangle$. Then, by Lemma~\ref{loc13}, $J\cap H =I$. As, $I$ is $k$-irreducible, $J=I$ or $H=I$. Therefore, $J/I=\langle \overline{0}\rangle$, or $H/I = \langle \overline{0}\rangle$.
    
Conversely, suppose $\langle \overline{0}\rangle$ in $S/I$ is $k$-irreducible. Suppose, $J \cap H =I$, where $J$, $H$ are $k$-ideals of $S$. Then $J/I$ and $H/I$ are also $k$-ideals of $S$ such that $J/I \cap H/I = \langle \overline{0}\rangle$. So, $J/I = \langle \overline{0}\rangle$ or $H/I=\langle \overline{0}\rangle$. If $J/I=\langle \overline{0}\rangle$, we have  $a+I=I$, for all $a \in J$. Therefore, $a \in I$. So, $J \subseteq I$. Now, if $y \in I$, then $y+I=I$. This implies $y+I \in J/I$. So, $y+i_1 = j+i_2$, for some $j \in J$ and $i_1$, $i_2 \in I$. Therefore, $y \in J$ as $J$ is a $k$-ideal and $I \subseteq J$. Thus $J=I$. Similarly, if $H/I = \langle \overline{0}\rangle$, then we have $K=I$. 
\end{proof}

\section{Under localisation}\label{Section 5: k-ideals under localisation}

Our aim in this section is to study the behaviour of the class of $k$-ideals and its subclasses under localisation. 
Let $S$ be a semiring and let $T$ be a multiplicatively closed subset of $S$. For each $k$-ideal $I$
of the semiring $T\inv{S}$, we consider
\begin{equation}
    I^c:= \{x \in S \mid x/1 \in I \} = I \cap S,\quad
\text{and}\quad  C:= \{I^c \mid I \;\text{is a $k$-ideal of}\; T\inv S\}. \tag{$\star$}\label{Equation of C}
\end{equation}

\begin{lemma}
    Let $S$ be a semiring and $T$ be a multiplicatively closed subset of $S$. If $I \in C$, then $(T\inv I)^c =I$.
\end{lemma}

\begin{proof}
    Let $x\in I$. Then $x/1 \in T\inv I$ implies $x \in (T\inv I)^c$. Therefore, $I \subseteq (T\inv I)^c$.
For the reverse inclusion, suppose $y \in (T\inv I)^c$. As $I \in C$, we have $I=J^c$, for some $k$-ideal $J$ of $T\inv S$. Now $y \in (T\inv I)^c $ implies $\frac{y}{1}=\frac{i}{s}$, for some $i\in I$ and $s\in T$. Then $ys=i \in I$, which implies $ys \in I=J^c$. Therefore $\frac{ys}{1} \in J$. It follows that $\frac{y}{1}=(\frac{1}{s})(\frac{ys}{1})\in J$. So, $y \in J^c=I$. Hence $(T\inv I)^c=I$. 
\end{proof}
\begin{lemma}\label{loc5}
    Let $S$ be a semiring and $T$ be a multiplicatively closed subset of $S$. Then 
    \begin{enumerate}
        \item If $I$ is a $k$-ideal of $T\inv S$, then $I^c$ is a $k$-ideal of $S$,
        \item If $I$ is a $k$-ideal of $S$, then $T\inv I$ is a $k$-ideal of $T\inv S$.
    \end{enumerate}
    
\end{lemma}

\begin{proof}
    (1) Let $x$, $x+y \in I^c$. Then $\frac{x}{1},\ \frac{x+y}{1} \in I$. This implies $\frac{x}{1}+\frac{y}{1}\in I$, which further implies $\frac{y}{1} \in I$ as $I$ is a $k$-ideal. Therefore $y\in I^c$.
    
    (2) Let $\frac{x}{s_1}$, $\frac{x}{s_1}+\frac{y}{s_2} \in T\inv I$. Then $\frac{xs_2+ys_1}{s_1s_2}=\frac{i}{s}$, for some $i\in I$ and $s\in T$. This implies $xs_2s+ys_1s=is_1s_2\in I$. Now, $\frac{x}{s_1}=\frac{i'}{s'}$, for some $i'\in I$ and $s'\in T$, which implies $xs'\in I$, and so, $xs_2ss'\in I$. Therefore, $xs_2ss'+ys_1ss'\in I$ implying $ys_1ss'\in I$, as $I$ is a $k$-ideal. Then $\frac{y}{s_2}=\frac{ys_1ss'}{s_2s_1ss'}\in T\inv I$, \textit{i.e.},  $\frac{y}{s_2} \in T\inv I$. Hence, $T\inv I$ is a $k$-ideal of $T\inv S$.
\end{proof}
In \cite[Theorem 3.1]{Azizi08}, Azizi established a correspondence between proper strongly irreducible ideals of a localized ring and proper strongly irreducible ideals of the ring itself. The next theorem extends this correspondence to the setting of semirings by considering subtractive ideals, thereby generalizing Azizi's result beyond the classical ring-theoretic framework.
\begin{theorem} \label{loc}
Let $S$  be a semiring and $T$ be a multiplicatively closed subset of $S$. Then
there is a one-to-one correspondence between the $k$-proper strongly irreducible ideals of $T\inv{S}$ 
and $k$-proper strongly irreducible ideals of $S$  contained in $C$ (defined in (\ref{Equation of C})) which do not meet $T$.
\end{theorem}

\begin{proof}
    Suppose $I$ is a  $k$-proper strongly irreducible ideal of $T\inv S$. By Lemma~\ref{loc5}, $I^c$ is a $k$-ideal of $S$ and $I^c \cap T=\varnothing $ since $I$ cannot contain a unit. Let $A \cap B \subseteq I^c$ for $k$-ideals $A$, $B$ of $S$. Then it follows that \[(T\inv A) \cap (T\inv B) =T\inv (A\cap B) \subseteq T\inv I^c =I.\] Therefore, $T\inv A \subseteq I$ or $T\inv B \subseteq I$. Hence $A\subseteq (T\inv A)^c \subseteq I^c$ or $B\subseteq (T\inv B)^c \subseteq I^c$. 
    Conversely, suppose $I$ is a $k$-proper strongly irreducible ideal of $S$ with $I\in C$ and $I \cap T =\varnothing $.  Let $A \cap B \subseteq T\inv I$, where $A$ and $B$ are $k$-ideals of $T\inv S$. Then $A^c \cap B^c =(A \cap B)^c \subseteq (T\inv I)^c$. Note that $I=(T\inv I)^c $. Thus $A^c \cap B^c \subseteq I$, and hence, $A^c \subseteq I$ or $B^c \subseteq I$. We then have $A=T\inv (A^c) \subseteq  T\inv I$ or $B=T\inv(B^c) \subseteq T\inv I.$
\end{proof}

Motivated by Azizi's results \cite[Theorem 3.4]{Azizi08} characterizing the relationships among strongly irreducible, prime, and maximal ideals in localized rings, we extend these ideas to the semiring context by formulating their subtractive counterparts in localized semirings.

\begin{proposition}
 If $S$ is a semiring, then the following are equivalent.
\begin{enumerate}
\item  Every $k$-primary ideal of $S$ is a $k$-strongly irreducible ideal.

\item For any $k$-prime ideal $P$ of $S$, every $k$-primary ideal of $S_P$ is a $k$-strongly irreducible
ideal.

\item For any $k$-maximal ideal $P$ of $S$, every $k$-primary ideal of $S_P$ is a $k$-strongly irreducible
ideal.
\end{enumerate}
\end{proposition}

\begin{proof}
	(1)$\Rightarrow$(2):
Let $I$ be a $k$-primary ideal of $S_P$. Then $I^c$ is a $k$-ideal. Let $xy\in I^c$. This implies $\frac{xy}{1}\in I$, which further implies $\frac{x}{1}\cdot \frac{y}{1} \in I$. It follows that, $\frac{x}{1} \in I$ or $(\frac{y}{1})^n\in I$. If $\frac{x}{1} \in I$, we have $x \in I^c$, and if $\frac{y^n}{1}=(\frac{y}{1})^n \in I$, we get $y^n\in I^c$. Now, $I^c\cap (S\setminus P)=\emptyset$, $I^c\in C$. By our assumption, $I^c$ is a $k$-strongly irreducible ideal of $S_P$. Therefore, by Theorem~\ref{loc}, we obtain $I=T\inv I^c$, where $T=S\setminus P$, implying that $I$ is a $k$-strongly irreducible ideal of $S_P$.

    (2)$\Rightarrow$(3):
Since every $k$-maximal ideal is $k$-prime, the proof is clear.

    (3)$\Rightarrow$(1): 
Suppose, $I$ is a $k$-primary ideal of $S$ and let $M$ be a $k$-maximal ideal of $S$. Then by Corollary~\ref{loc4}, $I_M=T\inv I$, where $T=S\setminus M$, is a $k$-primary ideal of $S_M$. Then by our assumption, $I_M$ is a $k$-strongly irreducible ideal of $S$. As $I$ is a $k$-primary ideal of $S$, we have $(I_M)^c=I$. Thus, $I$ is  $k$-strongly irreducible.
\end{proof}
The next lemma provides a structural blueprint of $k$-ideals under localisation. 

\begin{lemma}\label{loc8}
    Let $T$ be a multiplicatively closed subset of a semiring $S$ and $f\colon S\to T\inv S$ be the natural homomorphism defined by  $f(r):=\frac{r}{1}$. If $I$ is a $k$-ideal of $S$, then $I^{ec} = \bigcup\limits_{s\in T}(I:s).$ Hence $I^e=\langle1\rangle$ if and only if $I$ meets $T$.
\end{lemma}

\begin{proof}
    If $I$ is a $k$-ideal of $S$, then  \[I^e=\mathcal{C}_k(\langle f(I) \rangle)=\mathcal{C}_k(T\inv I)=T\inv I.\] Thus \[x\in I^{ec}=(T\inv I)^{c}\iff\frac{x}{1}\in T\inv I\iff\frac{x}{1}=\frac{i}{s}\iff xs=i\in I\iff x\in \bigcup_{s\in T}(I:s),\]  for some $i\in I$, $s\in T$.
Next, let  $I^{e}=\langle1 \rangle = T\inv S$. This implies $1\in I^{e}$. Then $1=\frac{1}{1} \in I^{ec} = \bigcup_{s\in T}$$(I:s)$. Therefore $1\cdot s'\in I$, for some $s'\in T$. So, $I\cap T\neq \emptyset$. Thus $I$ meets $T$. Conversely, suppose $x\in I\cap T$.Then $\frac{x}{1} \in T\inv I=I^{e}$. So, $1=(\frac{1}{x})\cdot(\frac{x}{1})\in I^{e}$. Therefore we get $I^{e}=\langle 1\rangle$.
\end{proof}

\begin{proposition}\label{contraction of k-prime ideals}
Let $\varphi\colon S\to S'$ be a semiring homomorphism and $\varphi_{!} \colon \mathrm{Spec}_k(S') \to \mathrm{Spec}_k(S)$ be the associated map of $\varphi$ defined by $\varphi_!(Q):=Q^c$. Let $P$ and $Q$ respectively be $k$-prime ideals of $S$ and $S'$. Then 
\begin{enumerate}
\item $P$ is the $k$-contraction of a $k$-prime ideal of $S'$ if and only if $P^{ec}=P$.
		
\item   $P$ is a $k$-contracted ideal if and only if  $\varphi_!$ is surjective.
		
\item $Q$ is a $k$-extended ideal implies that $\varphi_!$ is injective.
\end{enumerate}
\end{proposition}

\begin{proof}
(1) Let $P=Q^{c}$, where $Q$ is a $k$-prime ideal of $S'$. Then $P=Q^{c}=\varphi\inv(Q)$. Therefore $\varphi(P)=Q$. Obviously $P\subseteq P^{ec}$. For the reverse inclusion, let $x\in P^{ec}$. This implies $x\in \varphi\inv(P^e)$, which further implies $\varphi(x)\in P^{e}=\mathcal{C}_k(\langle \varphi(P)\rangle)$. Then $\varphi(x)+r\in \langle\varphi(P)\rangle$, for some $r\in\langle\varphi(P)\rangle=\langle Q \rangle$. Therefore $r=\sum\limits_{i=1}^n r_iq_i$, where $r_i\in S'$ and $q_i\in Q$, for all $i=1,2,\ldots,n$. Since $q_i\in Q$, $r_i\in S'$, we have $r=\sum\limits_{i=1}^n r_iq_i\in Q$. So, $\varphi(x)+r\in Q$. This implies $\varphi(x)\in Q$, as $Q$ is a $k$-ideal of $S'$. Therefore $x\in \varphi\inv(Q)=Q^c=P$. Thus $P^{ec}=P$.
Conversely, suppose $P^{ec}=P$. Let $T$ be the image of $S\setminus P$ in $S'$. Now \[S\setminus P=S\setminus P^{ec}=S\setminus \varphi\inv(P^e)=\varphi\inv(S'\setminus P^e).\] Therefore, $P^e$ does not meet $T$. Since $P$ is a $k$-ideal of $S$, $P^e$ is a $k$-ideal of $S'$. Then by Lemma~\ref{loc8}, $(P^e)^e$ is a proper $k$-ideal of $T\inv S'$. So, it is contained in a $k$-maximal ideal $M$ of $T\inv S'$, \textit{i.e.}, $(P^e)^e$ is contained in a $k$-prime ideal $M$ of $T\inv S'$. Let $Q$ be the contraction of $M$ in $S'$. Then $Q=M^c$. Therefore, $Q$ is $k$-prime, as $M$ is $k$-prime. Now, \[P^e\subseteq (P^e)^{ec}\subseteq M^c=Q.\] Let $r\in Q\cap T$, if possible. Then $r\in Q=M^c$, which implies $\frac{r}{1}\in M$. This further implies $1=(\frac{1}{r})\cdot(\frac{r}{1})\in M$, which is impossible. Therefore $Q\cap T=\emptyset$. Since $P^e\subseteq Q$, we have $P^{ec}\subseteq Q^c$, which gives $P\subseteq Q^c$. If there exists some $x\in Q^c$ such that $x\notin P$, then $\varphi(x)\in Q$ and $\varphi(x)\in \varphi(S\setminus P)=T$. This contradicts the fact that $Q\cap T=\emptyset$. Thus, $P=Q^c$.

 (2) Let $P\in \mathrm{Spec}_k(S)$. Suppose $P$ is a $k$-contracted ideal. Therefore $P=I^c$, for some $k$-ideal $I$ of $S'$. Now \[P^{ec}=I^{cec}=I^c=P.\] This implies $P$ is the $k$-contraction of a $k$-prime ideal, say $J$. Then $J\in \mathrm{Spec}_k(S')$ and $J^c=P$. Thus, $\varphi_!(J)=J^c=P$, that is,  $\varphi_!$ is surjective.
Conversely, suppose that $\varphi_!$ is surjective. Let $P$ be a $k$-prime ideal of $S$. Then there exists $Q\in \mathrm{Spec}_k(S')$ such that $\varphi_!(Q)=P$, which implies $Q^c=P$. Hence, $P$ is a $k$-contracted ideal.

 (3) Let every $k$-prime ideal of $S'$ is a $k$-extended ideal. If possible, let there exist $Q_1$, $Q_2\in \mathrm{Spec}_k(S')$ such that $\varphi_!(Q_1)=\varphi_!(Q_2)$. This implies $Q_1^c=Q_2^c$. Now, $Q_1=J_1^e$ and $Q_2=J_2^e$, for some $k$-ideals $J_1$, $J_2$ of $S$. $Q_1^c=Q_2^c$ implies $Q_1^{ce}=Q_2^{ce}$, which further implies $J_1^{ece}=J_2^{ece}$. So, $J_1^e=J_2^e$. Therefore $Q_1=Q_2$. Thus $\varphi_!$ is injective.	
\end{proof}

As an extension to the results established by Atani in \cite[Lemma 9]{Atani08} in the study of localisation of semirings, we examine the behaviour of specific classes of subtractive ideals under the operation of taking $k$-radicals. 

\begin{proposition}\label{loc6}
Let $I$ be a $k$-primary ideal of a semiring $S$ with $\mathcal{R}_k(I) = P$. Then the following hold:
\begin{enumerate}
	
\item $IS_T$ is a $k$-primary ideal of $S_T$.

\item If $P \cap T = \emptyset$, then $IS_T \cap S = I$.

\item  If $P \cap T = \emptyset$, then $(IS_T :_{S_T} PS_T)=(I :_S P)S_T$.

\item If $P \cap T = \emptyset$, and $J$ is a $k$-ideal of $S$ such that $JS_T\subseteq IS_T$, then $J \subseteq  I.$
\end{enumerate}
\end{proposition}

\begin{proof}
(1) If $I$ is a $k$-primary ideal of $S$, then $IS_T$ is a $k$-ideal of $S_T$, by Lemma~\ref{loc5}. Let $(\frac{x}{s_1})\cdot(\frac{y}{s_2})\in IS_T$, where $\frac{x}{s_1}$, $\frac{y}{s_2}\in S_T$. Then $\frac{xy}{s_1s_2}=\frac{i}{s}$, for some $i\in I$, $s\in T$. So, $xys\in I$. Since $I$ is a $k$-primary ideal, either $xs\in I$ or $y^n\in I$, for some natural number $n$. If $xs\in I$, then $\frac{x}{s_1}=\frac{xs}{ss_1}\in IS_T$. If $y^n\in I$, then $(\frac{y}{s_2})^n=\frac{y^n}{s_2^n}\in IS_T$.
    
(2) Clearly, we have $I\subseteq IS_T\cap S$. For the reverse inclusion, let $x\in IS_T\cap S$. Then $\frac{x}{1}\in IS_T$. Therefore, there exists some $i\in I$ and $s\in T$ such that $\frac{x}{1}=\frac{i}{s}$. This implies $xs=i\in I$. Since $I$ is $k$-primary, either $x\in I$ or $s^n\in I$ for some natural number $n$. Now, $s^n\in I$ implies $s\in \mathcal{R}_k(I)=P$, which is a contradiction. So, $x\in I$.
    
(3) Let $\frac{x}{s}\in (I:_SP)$. Then $xP\subseteq I$. Therefore, for all $\frac{p}{s_1}\in PS_T$, we have $\frac{xp}{s_1s}\in IS_T$. This implies that $(\frac{x}{s}).(\frac{p}{s_1})\in IS_T$, which further implies, $\frac{x}{s}P_{S_T}\subseteq IS_T$. So, $\frac{x}{s}\in (IS_T:_{S_T} PS_T)$. We have to show $xy\in I$, for all $y\in P$. Now, for all $y\in P$, clearly $$(\frac{x}{s})\cdot(\frac{y}{1})=\frac{xy}{s} \in IS_T.$$ Then, there exist $a\in I$, $b\in T$ such that $\frac{xy}{s}=\frac{a}{b}$. Therefore $(xy)b\in I$. As $I$ is a $k$-primary ideal of $S$, either $xy\in I$ or $b^n\in I$, for some natural number $n$. If $b^n\in I$, we have $b\in \mathcal{R}_k(I)=P$, a contradiction. Thus $xy\in I$. Hence the proof.
    
(4) Let $x\in J$. Since $\frac{x}{1}\in JS_T\subseteq IS_T$, there exist $i\in I$ and $s\in T$ such that $\frac{x}{1}=\frac{i}{s}$, which implies $xs\in I$. Therefore, $I$ is a $k$-primary ideal implies $x\in I$ or $s^n\in I$. If $s^n\in I$, then we have $s\in \mathcal{R}_k(I)=P$, a contradiction to the fact that $P\cap T=\emptyset$. Therefore, $x\in I$. Hence $J\subseteq I$.
\end{proof}

\begin{proposition}\label{loc7}
Let $I$ be a $k$-ideal of a semiring $S$. If $IS_T$ is a $k$-strongly irreducible
ideal of $S_T$, then $IS_T\cap S$ is a $k$-strongly irreducible ideal of $S$.
\end{proposition}

\begin{proof}
Let $IS_T$ be a $k$-strongly irreducible ideal of $S_T$. Suppose $A\cap B\subseteq IS_T\cap S$, where $A$, $B$ are two $k$-ideals of $S$. Then \[(A\cap B)S_T\subseteq (IS_T\cap S)S_T=IS_T,\] by \cite[Lemma 4]{Atani08}. This implies $AS_T\cap BS_T\subseteq IS_T$. So, $AS_T\subseteq IS_T$ or $BS_T\subseteq IS_T$, as $IS_T$ is $k$-strongly irreducible. Therefore, $AS_T\cap S\subseteq IS_T\cap S$ or $BS_T\cap S\subseteq IS_T\cap S$, which further implies, $A\subseteq AS_T\cap S\subseteq IS_T\cap S$ or $B\subseteq BS_T\cap S\subseteq IS_T\cap S$,  \textit{i.e.}, $A\subseteq IS_T\cap S$ or $B\subseteq IS_T\cap S$. 	
\end{proof}

\begin{theorem}
Let $I$ be a $k$-strongly irreducible $k$-primary ideal of a semiring $S$ such that
$\mathcal{R}_k(I) \cap T = \emptyset$. Then $IS_T$ is $k$-strongly irreducible.
\end{theorem}

\begin{proof}
Let $I$ be a $k$-strongly irreducible primary ideal of $S$. Suppose $A$, $B$ be two $k$-ideals of $S_T$ such that $A\cap B \subseteq IS_T$. Then \[(A\cap S)\cap (B\cap S)\subseteq IS_T\cap S=I,\] by Proposition~\ref{loc6}. Let $x$, $x+y\in A\cap S$. Then $\frac{x}{1}$, $\frac{x+y}{1}\in A$. This implies $\frac{x}{1}+\frac{y}{1} \in A$. Therefore $\frac{y}{1}\in A$ as $A$ is a $k$-ideal of $S_T$, which implies $y\in A\cap S$. So, $A\cap S$ and $B\cap S$ are $k$-ideals of $S$. Since $I$ is $k$-strongly irreducible $A\cap S\subseteq I$ or $B\cap S\subseteq I$. Thus $A=(A\cap S)S_T\subseteq IS_T$ or $B=(B\cap S)S_T\subseteq IS_T$.		
\end{proof}

\begin{theorem}
Assume that $I$ is a $k$-primary ideal of a semiring $S$ with $\mathcal{R}_k(I) \cap T = \emptyset$
and let $IS_T$ be $k$-strongly irreducible ideal of $S_T$. Then $I$ is $k$-strongly irreducible.
\end{theorem}

\begin{proof}
Let $IS_T$ be a $k$-strongly irreducible ideal of $S_T$. Since $I$ is a $k$-primary ideal such that $\mathcal{R}_k(I)\cap T=\emptyset$, by Proposition~\ref{loc6} we have $IS_T\cap S=I$. Now, by Theorem~\ref{loc7}, it follows that $I=IS_T\cap S$ is a $k$-strongly irreducible ideal of $S$.		
\end{proof}









\section{In special semirings}\label{Section 6: k-ideals and its variants in special semirings}
This section is bifurcated in scope. First, we investigate the behaviour of 
$k$-ideals across several distinguished  classes of semirings, including Laskerian semirings, von Neumann regular semirings, subtractive unique factorisation semidomains, principal ideal semidomains, and weakly Noetherian semirings. The subsequent subsection is devoted to a focused analysis of 
$k$-ideals in the specialized setting of additively idempotent semirings.

We begin by recalling one important class of semirings, obtained as natural generalizations of the corresponding ring-theoretic notions.
\begin{definition}
A semiring $S$ is said to be \emph{Laskerian} if every $k$-ideal has a $k$-primary decomposition.

\end{definition}
In the results that follow, we analyse the interplay between $k$-strongly irreducible ideals and other distinguished subclasses of $k$-ideals, notably $k$-prime and $k$-primary ideals, across a range of semiring classes.

\begin{proposition}\label{loc3}
Suppose $S$ is a semiring. 
\begin{enumerate}
\item If $I$ is a $k$-strongly irreducible ideal of $S$, then $I$ is a $k$-prime ideal of $S$ if and only if $I=\mathcal{R}_k(I)$.
\item If $S$ is a Laskerian semiring, then every $k$-strongly irreducible ideal is a $k$-primary ideal.

\item If $S$ is a von Neumann regular semiring, then an ideal is $k$-strongly irreducible if and only if
it is a $k$-primary ideal.
\end{enumerate}
\end{proposition}

\begin{proof}
	(1) If $I$ is $k$-prime, then clearly $I=\mathcal{R}_k(I)$. Conversely, suppose $I=\mathcal{R}_k(I)$. Let $AB \subseteq I$, where $A$, $B$ are two $k$-ideals of $S$. Then \[A\cap B \subseteq \mathcal{R}_k(A\cap B)=\mathcal{R}_k(A)\cap \mathcal{R}_k(B)=\mathcal{R}_k(AB)\subseteq \mathcal{R}_k(I)=I.\] Therefore, $A \subseteq I$ or $B \subseteq I$, as $I$ is $k$-strongly irreducible. Hence $I$ is $k$-prime.
    
        (2) Let $I$ be a $k$-strongly irreducible ideal of $S$ and suppose $\bigcap\limits_{i=1}^n A_i$ be a $k$-primary decomposition  of $I$. Then $\bigcap\limits_{i=1}^n A_i \subseteq I$. As $I$ is $k$-strongly irreducible, we have $A
        _j \subseteq I$, for some $j\in \{1,2,...,n\}$. This implies $A_j \subseteq I \subseteq A_j$, which further implies $I=A_j$. Thus $I$ is $k$-primary.
        
        (3) In a von Neumann regular semiring, $k$-prime and $k$-primary ideals coincide by Theorem~\ref{loc14}. Let $I$ be $k$-strongly irreducible and suppose $ab\in I$. It is enough, by Theorem~\ref{loc11}, to prove
\[
\mathcal C_k(\langle a\rangle)\cap\mathcal C_k(\langle b\rangle)\subseteq I.
\]
Take $c$ in the intersection. Choose $\alpha_1,\alpha_2,\beta_1,\beta_2\in S$ such that
\[
c+a\alpha_1=a\alpha_2,\qquad c+b\beta_1=b\beta_2.
\]
Put $K=\mathcal C_k(\langle ab\rangle)$. Multiplying the first equality by $b\beta_1$ and the second by $a\alpha_1$ shows, respectively, that $cb\beta_1\in K$ and $ca\alpha_1\in K$. Now multiplying the above equalities gives
\[
c^2+cb\beta_1+ca\alpha_1+ab\alpha_1\beta_1=ab\alpha_2\beta_2.
\]
By subtractivity therefore we get $c^2\in K$. Since $S$ is von Neumann regular, $c=c^2r$ for some $r\in S$, so $c\in K\subseteq I$. Thus Theorem~\ref{loc11} gives $a\in I$ or $b\in I$, and $I$ is $k$-prime, hence $k$-primary. Conversely, every $k$-primary ideal is $k$-prime by Theorem~\ref{loc14}, and every $k$-prime ideal is $k$-strongly irreducible.
\end{proof}

\begin{remark}
    In a bounded distributive lattice, each $k$-ideal is a $k$-radical ideal. Therefore, by Proposition~\ref{loc3}, every $k$-strongly irreducible ideal is $k$-prime. This confirms the result proved in \cite[Theorem 5]{Iseki56}.
\end{remark}

For rings, Azizi examined various properties of strongly irreducible ideals in unique factorisation domains \cite[Theorem 2.2]{Azizi08}. We now present subtractive analogues of these properties, extending the discussion to the present framework. The proof of the following lemma is identical to the corresponding result (dropping the `$k$-part') for rings.

\begin{lemma}\label{loc15}
    Let $S$ be a subtractive unique factorisation semidomain and let $I$ be a proper ideal of $S$. Then:
    \begin{enumerate}
        \item $I$ is $k$-strongly irreducible if and only if, whenever
        $p_1^{n_1}\cdots p_r^{n_r}\in I$ with distinct prime elements $p_i$ and $n_i\ge1$, one has $p_j^{n_j}\in I$ for some $j$.
        \item If $I$ is nonzero and principal, then $I$ is $k$-strongly irreducible if and only if a generator of $I$ is associated to a power of a prime element.
    \end{enumerate}
\end{lemma}





\begin{theorem}\label{loc2}
Let $S$ be a subtractive unique factorisation semidomain and let $I$ be a proper ideal of $S$.
\begin{enumerate}
\item $I$ is $k$-strongly irreducible if and only if, for all $x,y\in S$, $\operatorname{lcm}(x,y)\in I$ implies $x\in I$ or $y\in I$.
\item Every $k$-strongly irreducible ideal is $k$-primary.
\end{enumerate}
\end{theorem}

\begin{proof}

(1) Suppose first that $I$ is strongly irreducible and
$z=\operatorname{lcm}(x,y)\in I$. Since $S$ is subtractive, the principal
ideals $\langle x\rangle$ and $\langle y\rangle$ are $k$-ideals, and
\[
\langle x\rangle\cap\langle y\rangle=\langle z\rangle\subseteq I.
\]
Strong irreducibility therefore gives $x\in I$ or $y\in I$. Conversely,
suppose the stated l.c.m.\ condition holds. If $A\cap B\subseteq I$ for
$k$-ideals $A$ and $B$, and neither $A$ nor $B$ is contained in $I$, choose
$x\in A\setminus I$ and $y\in B\setminus I$. Then
\[
\operatorname{lcm}(x,y)\in A\cap B\subseteq I,
\]
contrary to the l.c.m.\ condition. Thus $I$ is strongly irreducible.

(2) Suppose that $ab\in I$ and $a\notin I$. If $a$ or $b$ is zero or a unit, the assertion is immediate, so assume that both are nonzero nonunits. Factor $a$ and $b$ into primes. By Lemma~\ref{loc15}(1), some prime power occurring in the factorisation of $ab$ belongs to $I$. It cannot involve only factors of $a$, since then $a\in I$, and if it involves only factors of $b$ then $b\in I$. Otherwise, for some common prime $p$, if its exponents in $a$ and $b$ are $\alpha$ and $\beta>0$, respectively, then $p^{\alpha+\beta}\in I$. Choose $N$ such that $N\beta\geq\alpha+\beta$. Since $p^{N\beta}\mid b^N$, we obtain $b^N\in I$. Hence $I$ is $k$-primary.
\end{proof}

As a consequence of the above results, we have the following corollaries.

\begin{corollary}
Let $S$ be a subtractive unique factorisation semidomain.
\begin{enumerate}

\item Every $k$-principal ideal of $S$ is a $k$-strongly irreducible ideal if and only if it is a
$k$-primary ideal.
	
\item Every $k$-strongly irreducible ideal of $S$ can be generated by a set of prime powers.
\end{enumerate}
\end{corollary}
\begin{proof}
For (1), the forward implication follows from Theorem~\ref{loc2}(2).
Conversely, let $I=\langle a\rangle$ be a nonzero principal $k$-primary ideal. If $a$ involves two distinct primes, write $a=xy$, where $x$ contains one prime-power factor and $y$ contains the remaining factors. Then $xy\in I$ and $x\notin I$, whereas no power of $y$ is divisible by $a$,
contradicting the primaryness of $I$. Thus $a$ is associated with a prime power, and Lemma~\ref{loc15}(2) gives strong irreducibility. The zero ideal is handled directly, since $S$ is a semidomain.

For (2), let $\mathcal P$ be the set of prime powers contained in $I$. Then
$\langle\mathcal P\rangle\subseteq I$. Conversely, if $0\ne x\in I$, Lemma~\ref{loc15}(1), applied to a prime factorisation of $x$, gives $p^m\in\mathcal P$ dividing $x$. Hence $x\in\langle\mathcal P\rangle$, and therefore
\[
I=\langle\mathcal P\rangle.
\]
\end{proof}



\begin{corollary}\label{loc4} 
Let $S$  be a semiring and let $T$ be a multiplicatively closed subset of $S$.

\begin{enumerate}
\item If $I$ is a $k$-strongly irreducible and a $k$-primary ideal of $S$ which does not meet $T$,
then $T\inv I$ is a $k$-strongly irreducible (and a $k$-primary) ideal of $T\inv S$.

\item If $S$ is a subtractive unique factorisation semidomain or a von Neumann regular semiring, then there is a one-to-one correspondence between the $k$-strongly irreducible ideals of $T\inv S$ and the
$k$-strongly irreducible ideals of $S$ which do not meet $T$.
\end{enumerate}
\end{corollary}

\begin{proof}
(1) We first show that if $I$ is a $k$-primary ideal of $S$ and $I\cap T=\emptyset$, then $(T\inv I)^c=I$. Clearly $I\subseteq (T\inv I)^c$. Let $x\in (T\inv I)^c$. Then $\frac{x}{1}\in T\inv I$. This implies $\frac{x}{1}=\frac{i}{s}$, for some $i\in I$ and $s\in T$, which further implies $xs=i\in I$. Let $x\notin I$. Then $s^n\in I$, for some natural number $n$. This implies $s^n\in I\cap T$, a contradiction. Therefore $x\in I$. Thus $(T\inv I)^c=I$, which implies $I$ is contained in $C$. Then it follows by Theorem~\ref{loc} that $T\inv I$ is a $k$-strongly irreducible ideal of $T\inv S$. It remains to show that $T\inv I$ is a $k$-primary ideal. Suppose $xy \in T\inv I$. There must exist $u_1$, $u_2 \in S$ and $s_1$, $s_2 \in T$ such that $x=\frac{u_1}{s_1} $ and $y=\frac{u_2}{s_2}$. Then $u_1 u_2 \in I$, and hence,  $u_1 \in I$ or $u_2^n \in I$. Thus $x\in T\inv I$ or $y^n \in T\inv I$.
    
(2) It follows from Theorem~\ref{loc2} and Proposition~\ref{loc3} that in a subtractive unique factorisation semidomain and in a von Neumann regular semiring, every $k$-strongly irreducible ideal is a $k$-primary ideal. Then, by (1), $T\inv I$ is a $k$-strongly irreducible ideal of $T\inv S$. Also, we have, by Theorem~\ref{loc}, for each $k$-strongly irreducible ideal $I$ of $T\inv S,\ I^c$ is a $k$-strongly irreducible ideal of $S$ that does not meet $T$.
\end{proof}

The converse of Lemma~\ref{loc16} fails to hold in general semirings. The following theorem shows that this converse is valid in the more restrictive setting of a \emph{principal ideal semidomain}.

\begin{theorem}
Each non-zero $k$-prime ideal of a principal ideal semidomain is $k$-maximal.
\end{theorem}

\begin{proof}
Let $S$ be a principal ideal semidomain and $P$ be a $k$-prime ideal of $S$. Then $P=\langle p \rangle$, for some element $p$ of $S$. Suppose $Q$ is a $k$-ideal of $S$ containing $P$. Then $Q=\langle q \rangle$, for some $q\in S$. Since $p\in Q$, we have $p=rq$, for some non-zero $r\in S$. As $P$ is $k$-prime, $rq\in P$ implies $r\in P$ or $q\in P$. If $q\in P=\langle p \rangle$, we have $Q=P$. If $r\in P=\langle p \rangle$, then $r=pr'=rqr'$, for some $r'\in S$. As $S$ is a semidomain and $r\neq 0$, $r=rqr'$ implies $qr'=1$, which further implies $Q=S$.	
\end{proof}

Recall that a partially ordered ring $(R,+,\cdot,\leqslant)$ is a ring with a partial order such that for all $a$, $b$, $c\in R$ and $a\leqslant b$ implies $a+c\leqslant b+c$, and $0\leqslant a$, $0\leqslant b$ implies $0\leqslant ab$. The positive cone of $(R,+,\cdot,\leqslant)$ is the set $R^+=\{x\in R\mid 0\leqslant x\}$. The positive cone of a partially ordered ring is always a \emph{cancellative} and \emph{conic} (or, \emph{zero-sum free}) semiring, that is, $a+b=a+c$ implies $b=c$ and $a+b=0$ implies $a=0$ and $b=0$.

\begin{lemma}\label{loc9}
    Let $S$ be a positive cone of a partially ordered ring $R$ such that $D(S)=R$, where $D(S):=\{a-b \mid a,b\in S\}$. Then any principal ideal of $S$ is a $k$-ideal.
\end{lemma}

\begin{proof}
    For an element $a\in S$, let $x+y\in \langle a\rangle$ and $x\in \langle a\rangle$. Therefore $x+y=at$ and $x=as$, for some $s$, $t\in S$. 
    Since $S$ is the positive cone of $R$, we have $0\leqslant x$, $y$ and $y\leqslant x+y$, which further implies that $y=a(t-s)\in D(S)$ and $0\leqslant t-s$. Therefore $t-s\in S$ and hence $y\in \langle a\rangle$. 
\end{proof}

The next theorem states an elementary property of a $k$-prime ideal in a \emph{unique factorisation semidomain}.

\begin{theorem}
In a \emph{UFSD} every non-zero $k$-prime
ideal of $S$ contains a prime element.
\end{theorem}

\begin{proof}
Let $S$ be a UFSD and $I$ be a non-zero $k$-prime ideal of $S$. Let $0\ne x\in I$ be a nonunit. Write $x=p_1^{n_1}\cdots p_r^{n_r}$ with distinct prime elements $p_i$ and $n_i\ge1$. Since $I$ is $k$-prime, it implies $p_i\in I$ for some $i$.
\end{proof}

The converse of the above theorem is true if $S$ is a cancellative, partially ordered, conic semiring.

\begin{theorem}
    If $S$ is a cancellative, partially ordered, conic semidomain and if every non-zero $k$-prime ideal of $S$ contains a prime element, then $S$ is a \emph{UFSD}.
\end{theorem}

\begin{proof}
Suppose $S$ is a cancellative, partially ordered, conic semidomain and every non-zero $k$-prime ideal of $S$ contains a prime element. We assume that $S$ is not a UFSD. Let \[T=\{u\in S \mid u \;\text{is a unit} \} \bigcup \{p_1p_2\cdots p_n\in S\}.\] Then $T$ is a multiplicatively closed saturated subset of $S$. Now, $S$ is a UFSD if and only if $T=S\setminus \{0\}$. As $S$ is not a UFSD, there exists a non-zero $x\in S$ such that $x\notin T$. By Lemma~\ref{loc9}, it is clear that $\langle x\rangle$ is a $k$-ideal. We claim that $\langle x\rangle \cap T=\emptyset$. As if $rx\in T$, for some $r\in S$, then $r\in T$. Therefore, $\langle x \rangle$ can be expanded to a $k$-prime ideal $I$ of $S$ which is disjoint from $T$. Then, by hypothesis, $I$ contains a prime element, a contradiction. This completes the proof.
\end{proof}

Let us recall that in a \emph{weakly Noetherian semiring} every set of $k$-ideals satisfies the ascending chain condition. The following proposition reveals an interesting structural property of $k$-irreducible ideals in \emph{weakly Noetherian semirings}. 

\begin{proposition}\label{weakly Noetherian semiring, every k-irreducible ideal is k-primary}
    Let $S$ be a weakly Noetherian semiring. Then every $k$-irreducible ideal is $k$-primary.
\end{proposition}

\begin{proof}
    Let $I$ be a $k$-irreducible ideal of $S$ and $ab\in I$ with $b\notin I$. For each $n\geqslant 1$, let $I_n:=\{r\in S \mid a^nr\in I\}$. Then each $I_n$ is a $k$-ideal of $S$ such that $I_n\subseteq I_{n+1}$, for all $n$. Therefore, there exists $m\in \mathds{N}$ such that $I_n=I_{n+1}$, for all $n\geqslant m$ as $S$ is weakly Noetherian. Let $A=\langle a^m\rangle+I$, $B=\langle b \rangle+I$. We show that $A\cap B=I$. Let $x\in A\cap B$. Then $x=a^mr+i_1$, for some $r\in S$ and $i_1\in I$. Also $x=bs+i_2$, for some $s\in S$ and $i_2\in I$. So, $$ax=abs+ai_2\in \langle ab \rangle+I.$$ Since $ab\in I$ $ax\in I$. Now $ax=a^{m+1}r+ai_1\in I$. Then $a^{m+1}r\in I$, as $I$ is a $k$-ideal. This implies $r\in I_{m+1}=I_m$, \textit{i.e.}, $a^mr\in I$, which further implies $x\in I$. Since $I$ is $k$-irreducible and $b\notin I$, we get $A=I$. Therefore, it follows that $a^m\in I$. Hence $I$ is $k$-primary.
\end{proof}

\smallskip
\subsection{Additively idempotent semirings}\label{Subsection: additively idempotent}

In this final segment, we investigate variants of $k$-ideals and $k$-congruences in the milieu of additively idempotent semirings. It is worth emphasizing that additively idempotent semirings may deviate substantially from the class of subtractive semirings (\textit{cf.} \cite[Example 6.39]{Golan}), thereby rendering the study of $k$-ideals particularly compelling.
This perspective enables us to strengthen several results of \cite{Lescot11, Lescot15}, where additively idempotent semirings are studied through the lens of \emph{$B_1$-algebras}. Here $B_1$ denotes the two-element Boolean algebra $\{0,1\}$, commonly referred to as the \emph{Boolean semifield}. We recall that, in additively idempotent semirings, the product of two $k$-ideals is again a $k$-ideal; see \cite[Lemma 3.27]{JunRayTolliver22}. This property fails in the broader class of general semirings; see \cite[Example 6.43]{Golan}.

The following definition is a subtractive version of the definition of cancellation ideals, first introduced by LaGrassa in her thesis \cite{LaGrassa95}.

\begin{definition}
A non-zero $k$-ideal $I$ of a semiring $S$ is called a \emph{$k$-cancellation ideal} if $IJ = IK$ implies $J = K$, for all $k$-ideals $J$ and $K$ of $S$.
\end{definition}
The following proposition characterizes $k$-cancellation ideals in terms of colon ideals and ideal inclusions in an additively idempotent setup. 
\begin{proposition}
Let $S$ be an additively idempotent semiring and $I$ be a non-zero $k$-ideal of $S$. Then the following statements are equivalent:
\begin{enumerate}
\item $I$ is a $k$-cancellation ideal of $S$.

\item $(IJ : I)=J$, for any $k$-ideal $J$ of $S$.

\item $IJ\subseteq IK$ implies $J\subseteq K$, for all $k$-ideals $J$, $K$ of $S$.
\end{enumerate}
\end{proposition}

\begin{proof}
Let $I,J,K$ be $k$-ideals of $S$.

(1)$\Rightarrow$(2): 
Obviously, $IJ\subseteq (IJ:I)$. Let $i_1x_1+i_2x_2+\cdots + i_nx_n\in I(IJ:I)$, $i_k\in I$, and $x_k\in I$, for all $k$. $x_kI\subseteq IJ$, for all $k$. Therefore, $\sum_{k=1}^n i_kx_k\in IJ$. So, $I(IJ:I)\subseteq IJ$. Now, for all $j\in J$, we have $jI\subseteq IJ$. Therefore $IJ\subseteq I(IJ:I)$. Thus $IJ=I(IJ:I)$. Since $S$ is an additively idempotent semiring, then by \cite[Lemma 3.27]{JunRayTolliver22}, $IJ$ is a $k$-ideal, furthermore by \cite[Proposition 2.4]{GoswamiDube24}, $(IJ:I)$ is a $k$-ideal.  Since $I$ is a $k$-cancellation ideal, we get $J=(IJ:I)$.

(2)$\Rightarrow$(3): 
Suppose, $IJ\subseteq IK$. Let $r\in (IJ:I)$. Then $rI\subseteq IJ\subseteq IK$. This implies $r\in (IK:I)$. So, $(IJ:I)\subseteq (IK:I)$, which further implies $J\subseteq K$.

(3)$\Rightarrow$(1):
Let $IJ=IK$. Then $IJ\subseteq IK, IK\subseteq IJ$, which implies $J\subseteq K $ and $K\subseteq J$. Therefore $J=K$.
\end{proof}
The next theorem provides a new characterization of $k$-prime ideals with respect to the natural partial order in any additively idempotent semiring.
\begin{theorem}\label{k-prime via natural partial order}
    Let $S$ be an additively idempotent semiring. Then a $k$-ideal $I$ is $k$-prime if and only if for $x$, $y \in S$ and $z \in I$, $xy \leqslant z$ implies that $x \in I$ or $y \in I$.
\end{theorem}

\begin{proof}
    Suppose, for all $x$, $y \in S$ and $z \in I$, $xy \leqslant z$ implies $x \in I$ or $y \in I$. We show that $I$ is $k$-prime. For that, let $y \in I$, $x+y \in I$. Then $x \leqslant x+y$, which implies $x\cdot1_S \leqslant x+y$. Therefore,  by assumption, $x\in I $ or $1_S \in I $. This implies $x \in I$ as $I$ is an ideal. Now, let $xy \in I$. Then $xy \leqslant xy$ implies $x \in I$ or $y \in I$. Conversely, suppose that $I$ is $k$-prime and for $x$, $y \in S$ and $z \in I$, $xy \leqslant z$. This implies $xy + z = z$, which further implies that $xy \in I$ as $I$ is a $k$-ideal. Therefore, $x \in I$ or $y \in I$.
\end{proof}

In similar spirit, the following theorem characterizes $k$-semiprime ideals with respect to the natural partial order of an additively idempotent semiring. 

\begin{theorem}\label{k-semiprime via natural partial order}
    Let $S$ be an additively idempotent semiring. A $k$-ideal $I$ is $k$-semiprime if and only if, for $x\in S$ and $z\in I$, $x^2\leqslant z$ implies $x\in I$.
\end{theorem}

\begin{proof}
Suppose $I$ is a $k$-semiprime ideal of $S$. Let $x \in S$ and $z \in I$ such that $x^2 \leqslant z$. Then $x^2 + z = z$. We get $x^2 \in I$ as $I$ is a $k$-ideal. Since $I$ is $k$-semiprime, we have $I=\mathcal{R}_k(I)$. Let $P$ be a $k$-prime ideal containing $I$. Then $x^2 \in P$ implies $x \in P$, which further implies $x \in \mathcal{R}_k(I) = I$. Conversely, suppose that for $x \in S$ and $z \in I$, $x^2 \leqslant z$ implies $x \in I$. Let $A$ be a $k$-ideal of $S$ such that $A^2 \subseteq I$. Let $x \in A$. Then $x^2 \in I$. Therefore, $x^2 \leqslant x^2$ implies $x \in I$, by our assumption. So, $x \in I$. We get $A \subseteq I$. Thus $I$ is $k$-semiprime.
\end{proof}
 
The  following theorem characterizes $k$-strongly irreducible ideals in an additively idempotent UFSD with respect to the natural partial order and the $\mathrm{l.c.m}$-operation. 
 
\begin{theorem}\label{k-strongly via order and lcm}
Let $S$ be an additively idempotent subtractive unique factorisation semidomain, and let $I$ be a proper ideal. Then $I$ is a $k$-strongly irreducible ideal in $S$ if and only if for $x$, $y \in S$ and $z \in I$, $\mathrm{l.c.m}\{x,y\} \leqslant z$ implies $x \in I$ or $y \in I$.
\end{theorem}

\begin{proof}
    Suppose, for all $x$, $y \in S$ and $z \in I$, $\mathrm{l.c.m}\{x,y\} \leqslant z$ implies $x \in I$ or $y \in I$. Let $\mathcal{C}_k(\langle x \rangle) \cap \mathcal{C}_k(\langle y \rangle) \subseteq I$. Then $\mathrm{l.c.m}\{x,y\} \in \langle x \rangle \cap \langle y \rangle \subseteq I$. This implies $\mathrm{l.c.m}\{x,y\} \in I$. Therefore $\mathrm{l.c.m}\{x,y\} \leqslant \mathrm{l.c.m}\{x,y\}$. So, we get $x \in I$ or $y \in I$, by our assumption. 
    Conversely, let $I$ be a $k$-strongly irreducible ideal of $S$. Suppose $x$, $y \in S$ and $z \in I$ such that $\mathrm{l.c.m}\{x,y\} \leqslant z$. Then $\mathrm{l.c.m}\{x,y\} + z = z$. We get $\mathrm{l.c.m}\{x,y\} \in I$, as $I$ is a $k$-ideal. Hence $x \in I$ or $y \in I$, by Theorem~\ref{loc2}(1).
\end{proof}

The rest of the paper is devoted to study various congruences and associated ideals in the additively idempotent setting. Let us now recall the definition of prime congruences from  \cite[Definition 3.4]{Lescot11}.

 \begin{definition}
     Suppose $S$ is an additively idempotent semiring. A congruence $\rho$ on $S$ is said to be \emph{prime} if $(ab) \rho 0$ implies $a \rho 0$ or $b \rho 0$. 
 \end{definition}
 
 The next definition is the counterpart of the Bourne congruence associated with any ideal in an additively idempotent semiring. 
 
 \begin{definition}[Section 3, \cite{Lescot11}]\label{Def: excellent congruences and saturated ideals}
     Let $S$ be a semiring and $J$ be an ideal of $S$. Then $\rho_J$ is the congruence on $S$ defined by $x\rho_Jy$ if and only if $x+z=y+z$ for some $z\in J$. This $\rho_J$ is called an \emph{excellent congruence}. Moreover, $J$ is called \emph{saturated} if $J=\bar J$, where $\bar J:=\{x\in S\mid x+z=z\text{ for some }z\in J\}$.
 \end{definition}

 \begin{remark}
 By a $k$-prime congruence, we mean a congruence that is a $k$-congruence as well as prime.    
 \end{remark}

     In \cite[Theorem 3.7]{Lescot11}, Lescot showed that the map $J\to \bar{J}$ is a closure operator. We shall prove that $\bar{J}$ coincides with the $\mathcal{C}_k(J)$. For that, we need the following lemma. 
 
 \begin{lemma}\label{loc10}
     Let $S$ be an additively idempotent semiring and $I$ be an ideal of $S$. Then $\bar{I}$ is a $k$-ideal of $S$.
 \end{lemma}
 
 \begin{proof}
Let $x,x+y\in\bar I$. Choose $z_1,z_2\in I$ with $x+z_1=z_1$ and $x+y+z_2=z_2$. Put $z=z_1+z_2\in I$. Using idempotence,
\[
y+z=y+z_1+z_2=y+z_1+(x+y+z_2)=x+y+z_1+z_2=z_1+z_2=z.
\]
Hence $y\in\bar I$, so $\bar I$ is a $k$-ideal.
 \end{proof}
 
 \begin{theorem}\label{saturated iff k-ideal}
     Let $S$ be an additively idempotent semiring. An ideal $I$ is saturated if and only if it is a $k$-ideal.
 \end{theorem}
 
 \begin{proof}
     If $I$ is saturated, then $I=\bar{I}$. By Lemma~\ref{loc10}, $I$ is a $k$-ideal. Conversely, suppose $I$ is a $k$-ideal. As for $x\in I$, $x+x=x$, we have $x\in \bar{I}$. Therefore $I\subseteq \bar{I}$. Let $x\in \bar{I}$. Then $x+z=z$, for some $z\in I$. This implies $x\in I$. Thus $I=\bar{I}$. 
 \end{proof}
The correspondence between $k$-ideals and $k$-congruences extends to some of their distinguished subclasses. In particular, we show that in an additively idempotent semiring there is a one-to-one correspondence between $k$-maximal congruences and $k$-maximal ideals. We begin by recalling the definition of a $k$-maximal congruence.
 \begin{definition}
     Let $S$ be a semiring. A $k$-congruence $\theta$ is said  to be \emph{$k$-maximal} if $\theta\subseteq \theta'$ implies either $\theta=\theta'$ or $\theta'=S\times S$, for any $k$-congruence $\theta'$ on $S$.
 \end{definition}
 The following theorem ensures the inclusion of the class of $k$-maximal congruences into the class of $k$-prime congruences, for any semiring $S$. 
  \begin{theorem}
    Let $S$ be a semiring. Then every $k$-maximal congruence is a $k$-prime congruence.
 \end{theorem}
 
 \begin{proof}
     Let $\rho=K_A$ be a $k$-maximal congruence for some ideal $A$ of $S$. Let $(uv,0)\in K_A$ but $(u,0)\notin K_A$. Then $u\notin A$, as $u+0=0+u$ will imply $uK_A0$. Consider the ideal $B=\langle u \rangle +A$. If $xK_Ay$, then we have $x+z_1=y+z_2$, for some $z_1$, $z_2\in A$, which implies $xK_By$. Therefore $K_A\subseteq K_B$. By maximality of $K_A$, $K_B=S\times S$. In particular, $0K_B1$. Then $0+us_1+a=1+us_2+b$, for some $s_1$, $s_2\in S$ and $a$, $b\in A$. So, $us_1 K_A(1+us_2)$. Now \[(uv)s_1=v(us_1)K_Av(1+us_2)=v+uvs_2.\] Since $uvK_A0$ we have \[0=0s_1K_A(uv)s_1K_Av+(uv)s_2K_Av+0\cdot s_2=v.\] Therefore, $vK_A0$.
 \end{proof}

 Han \cite[Theorem 5.4]{Han21} established a correspondence between $k$-prime ideals and $k$-prime congruences, and used this correspondence to show that the resulting spectra are homeomorphic. In the following theorem, we present a purely algebraic proof of the correspondence between $k$-maximal ideals and $k$-maximal congruences. 
 \begin{theorem}
 	
     Let $S$ be an additively idempotent semiring and $I$ be a $k$-ideal of $S$. Then $K_I$ is a $k$-maximal congruence if and only if $I$ is a $k$-maximal ideal.
 \end{theorem}
 
 \begin{proof}
     Suppose $I$ is a $k$-maximal ideal. Let $K_I\subseteq \rho$, for some $k$-congruence $\rho$ on $S$. Let $x\in I$. Then $x+x=x$. This implies $xK_I0$, which further implies $x\rho0$. So $$x\in I_\rho :=\{y\in S\mid y\rho0\}.$$ As $I$ is $k$-maximal either $I=I_\rho$ or $I_\rho=S$. Let $I=I_\rho$. Now as $\rho$ is a $k$-congruence $\rho=K_J$, for some ideal $J$ of $S$. Then $x\rho y$ implies $x+z_1=y+z_2$, for some $z_1$, $z_2\in J$. Since $z_1$, $z_2\in J$, we have $z_1\rho0$ and $z_2\rho0$. So, $z_1$, $z_2\in I_\rho=I$. Therefore $xK_Iy$, \textit{i.e.}, $\rho=K_I$. If $I_\rho=S$, we have for all $x\in S$, $x\rho0$. This implies $x\rho y$, for all $x$, $y\in S$. Thus $\rho=S\times S$.
Conversely, suppose $K_I$ is a $k$-maximal congruence and $I\subseteq J$, where $J$ is a $k$-ideal of $S$. Then $K_I\subseteq K_J$. This implies $K_I=K_J$ or $K_J=S\times S$. Let $K_I=K_J$. For $x\in J$, $x+x=x$ implies $xK_J0$, which further implies $xK_I0$. Then $x+z_1=z_2$, for some $z_1$, $z_2\in I$. So $x\in I$, that is, $I=J$. If $K_J=S\times S$ for all $x\in S$, $(x,0)\in K_J$. This implies $x+a=b$, for some $a$, $b\in J$. Therefore $x\in J$. Hence $J=S$.
 \end{proof}
 The next theorem concludes that the class of $k$-prime congruences is contained in the class of $k$-irreducible congruences in an additively idempotent semiring $S$. 
 
 \begin{theorem}
    Let $S$ be an additively idempotent semiring. Every $k$-prime congruence is $k$-irreducible.
\end{theorem}

\begin{proof}
Let $K_J$ be a $k$-prime congruence and suppose $K_A\cap K_B=K_J$, where $A,B,J$ are ideals. By Theorem~\ref{saturated iff k-ideal}, the zero classes of $K_A,K_B,K_J$ are respectively $\bar A,\bar B,\bar J$. Taking zero classes in the equality of congruences gives
\[
\bar A\cap\bar B=\bar J.
\]
If neither $\bar A$ nor $\bar B$ equals $\bar J$, choose $a\in\bar A\setminus\bar J$ and $b\in\bar B\setminus\bar J$. Since $\bar A$ and $\bar B$ are ideals, $ab\in\bar A\cap\bar B=\bar J$, so $ab\,K_J\,0$. Primality of $K_J$ gives $a\,K_J\,0$ or $b\,K_J\,0$, i.e. $a\in\bar J$ or $b\in\bar J$, a contradiction. Thus, say, $\bar A=\bar J$.

We now prove $K_A=K_J$. Since $K_J\subseteq K_A$, only the reverse inclusion is needed. If $xK_Ay$, then $x+a_1=y+a_2$ for some $a_1,a_2\in A\subseteq\bar A=\bar J$. Choose $j_1,j_2\in J$ with $a_1+j_1=j_1$ and $a_2+j_2=j_2$. Adding $j_1+j_2$ to the equality yields $x+j_1+j_2=y+j_1+j_2$, so $xK_Jy$. Hence $K_A=K_J$.
\end{proof}

 The accompanying result establishes that, in an additively idempotent semiring, the congruence associated with any $k$-irreducible ideal is itself $k$-irreducible. The validity of the converse implication, however, remains an open problem.
 \begin{theorem}\label{k-irreducible ideal to k-irreducible congruence}
    Let $S$ be an additively idempotent semiring. If $I$ is a $k$-irreducible ideal, then $K_I$ is a $k$-irreducible congruence.
\end{theorem}

\begin{proof}
Suppose $K_{A_1}\cap K_{A_2}=K_I$ for ideals $A_1,A_2$. Taking zero classes and using Theorem~\ref{saturated iff k-ideal} gives
\[
\mathcal C_k(A_1)\cap\mathcal C_k(A_2)=I,
\]
because the zero class of $K_A$ is $\mathcal C_k(A)=\bar A$ and $I=\bar I$. The two $k$-closures are $k$-ideals, so $k$-irreducibility of $I$ yields $\mathcal C_k(A_1)=I$ or $\mathcal C_k(A_2)=I$. Assume $\mathcal C_k(A_1)=I$. Then $A_1\subseteq I$, hence $K_{A_1}\subseteq K_I$; the reverse inclusion follows from $K_I=K_{A_1}\cap K_{A_2}$. Thus $K_{A_1}=K_I$, and similarly in the other case.
\end{proof}
 
 

We call an excellent congruence $\rho_J$ \emph{excellent-irreducible} if $\rho_{J_1}\cap\rho_{J_2}=\rho_J$ for excellent congruences $\rho_{J_1},\rho_{J_2}$ implies $\rho_{J_1}=\rho_J$ or $\rho_{J_2}=\rho_J$.

 The following theorem relaxes Theorem~\ref{k-irreducible ideal to k-irreducible congruence} by replacing $k$-irreducibility with the weaker notion of excellent irreducibility, arising from ideals whose saturated closures are $k$-irreducible.
 \begin{theorem}
    Let $S$ be an additively idempotent semiring and let $\bar{J}$ be a $k$-irreducible ideal of $S$. Then $\rho_J$ is an excellent-irreducible congruence on $S$.
 \end{theorem}
 
 \begin{proof}
     Let $\rho_{J_1}\cap \rho_{J_2}=\rho_J$, where $\rho_{J_1}, \rho_{J_2}$ are two excellent congruences on $S$. Define $I_{J_i}:=\{x\in S \mid x\rho_{J_i}0\}$ for $i=1,2$. Both $I_{J_1}, I_{J_2}$ are $k$-ideals of $S$. Let $x\in I_{J_1}\cap I_{J_2}$. Then $x\rho_{J_1}\cap \rho_{J_2}0$. This implies $x\rho_J0$. Therefore $x\in \bar{J}$. For the reverse inclusion, let $y\in \bar{J}$. Then there exists $z\in J$ such that $y+z=z$, which implies $y\rho_J0$, that is, $y\rho_{J_1}\cap \rho_{J_2}0$. Thus $y\in I_{J_1}\cap I_{J_2}$. Now, as $\bar{J}$ is $k$-irreducible, $I_{J_1}=\bar{J}$ or $I_{J_2}=\bar{J}$. Suppose $I_{J_1}=\bar{J}$. Let $x\in J_1$. Then $x+x=x$ implies $x\rho_{J_1}0$. Therefore $x\in I_{J_1}=\bar{J}$. If $x\rho_{J_1}y$, then there exists $a\in J_1$ such that $x+a=y+a$. Now $a\in \bar{J}$ implies $a+z=z$, for some $z\in J$. So, $x+a+z=y+a+z$, \textit{i.e.}, $x+z=y+z$, whence $x\rho_Jy$. Hence $\rho_{J_1}=\rho_J$.
 \end{proof}

Next, we recall from  \cite[Definition 3.10]{Lescot11}, a proper subclass of prime ideals and define its subtractive counterpart. 

 \begin{definition}
 An ideal $I$ of a semiring $S$ is said to be \emph{absolutely prime} if $I\neq S$ and $(ab)\rho_I(ac)$ implies $a\in \bar{I}$ or $b\rho_Ic$.
    An ideal $I$ is said to be \emph{absolutely $k$-prime} if it is a $k$-ideal and absolutely prime.   
 \end{definition}
 
 \begin{theorem}
     Let $S$ be an additively idempotent semiring. Then every absolutely $k$-prime ideal is $k$-prime.
 \end{theorem}
 
 \begin{proof}
     Let $I$ be an absolutely $k$-prime ideal and $ab\in I$. Then $ab+ab=ab$ implies $(ab)\rho_I0$, that is, $(ab)\rho_I(a\cdot 0)$, which further implies $b\rho_I0$ or $a\in \bar{I}=I$. If $b\rho_I0$ we have $b+z=z$, for some $z\in I$, whence $b\in I$. Hence $I$ is $k$-prime.
 \end{proof}
 

\section*{Acknowledgement}
The first author expresses sincere thanks to the Council of Scientific and Industrial Research (CSIR), Government of India, for the financial support (File No. 09/0096(17468)/2024-EMR-I). The third author is grateful to the University Grants Commission (India) for providing a Senior Research Fellowship (ID:
211610013222/ Joint CSIR-UGC NET JUNE 2021). The first, third, and fourth authors are also grateful for the support provided by the DST-FIST Purse Programme (Programme No.~SR/FST/MS-II/2021/101(C)) of the Department of Mathematics, Jadavpur University, Kolkata, India.

\section*{Declarations}
\begin{itemize}
	\item Conflict of interest: Not applicable.
	
	\item Research involving human participants and/or animals: Not applicable.
	
	\item Informed consent: Not applicable.
	
	\item Competing interest: Not applicable.
\end{itemize}

\end{document}